\documentclass[a4paper,10pt]{amsart}

\usepackage[utf8]{inputenc}
\usepackage[T1]{fontenc}
\usepackage[bookmarks]{hyperref}
\usepackage{marvosym}
\usepackage{wasysym}
\usepackage{latexsym}
\usepackage{amsfonts}
\usepackage{amsmath}
\usepackage{amssymb,amsxtra,amscd}
\usepackage{mathrsfs}
\usepackage{todonotes}
\usepackage{lmodern}

\newcommand{\R}{\mathbb{R}}
\newcommand{\C}{\mathbb{C}}

\newcommand{\N}{\mathbb{N}}

\newcommand{\K}{\mathbb{K}}
\newcommand{\Z}{\mathbb{Z}}
\newcommand{\supp}{\mbox{supp}\,}

\newtheorem{theorem}{Theorem}[section]

\newtheorem{corollary}[theorem]{Corollary}
\newtheorem{proposition}[theorem]{Proposition}
\newtheorem{problem}[theorem]{Problem}

\theoremstyle{definition}
\newtheorem{remark}[theorem]{Remark}
\newtheorem{definition}[theorem]{Definition}
\newtheorem{example}[theorem]{Example}

\begin{document}

\title[Dynamics of weighted composition operators]{Dynamics of weighted composition operators on function spaces defined by local properties}

\author[T. Kalmes]{Thomas Kalmes}
\address{Faculty of Mathematics\\ Chemnitz Technical University\\
	09107 Chemnitz, Germany}
\email{thomas.kalmes@math.tu-chemnitz.de}

\date{}

\begin{abstract}
	We study topological transitivity/hypercyclicity and topological (weak) mixing for weighted composition operators on locally convex spaces of scalar-valued functions which are defined by local properties. As main application of our general approach we characterize these dynamical properties for weighted composition operators on spaces of ultradifferentiable functions, both of Beurling and Roumieu type, and on spaces of zero solutions of elliptic partial differential equations. Special attention is given to eigenspaces of the Laplace operator and the Cauchy-Riemann operator, respectively. Moreover, we show that our abstract approach unifies existing results which characterize hypercyclicity, resp.\ topological mixing, of (weighted) composition operators on the space of holomorphic functions on a simply connected domain in the complex plane, on the space of smooth functions on an open subset of $\R^d$, as well as results characterizing topological transitiviy of such operators on the space of real analytic functions on an open subset of $\R^d$. 
\end{abstract}



\maketitle

\section{Introduction}

Dynamical properties like topological transitivity/hypercyclicity and to\-po\-lo\-gi\-cal (weak) mixing for weighted composition operators on various function spaces have been investigated by many authors in different settings. Recall, that an operator, i.e.\ a continuous linear self-map $T$ on a topological vector space $E$, is called topologically transitive, resp.\ topologically mixing, if for every pair of non-empty, open subsets $U,V$ of $E$ the sets $T^m(U)$ and $V$ intersect for some $m\in\N$, resp.\ for all sufficiently large $m\in\N$, whereas $T$ is called topologically weakly mixing if $T\oplus T$ is topologically transitive on $E\oplus E$. Moreover, $T$ is hypercyclic if there is $x\in E$ such that its orbit under $T$, i.e.\ the set $\{T^m x;\,m\in\N_0\}$, is dense in $E$. Clearly, every hypercyclic operator is topologically transitive, and the converse holds for operators on separable, complete, metrizable topological vector spaces by Birkhoff's Transitivity Criterion \cite[Theorem 2.19]{Grosse-ErdmannPeris}.

The most prominent setting for hypercyclic weighted composition operators is the space of holomorphic functions $\mathscr{H}(X)$ endowed with the compact-open topology on an open subset $X$ of the complex plane. Starting from Birkhoff's result \cite{Birkhoff} from which it follows that the translation operators $T_a(f)(\cdot):=f(\cdot+a), a\in\C\backslash\{0\},$ on the space of entire functions $\mathscr{H}(\C)$ are hypercyclic, many authors have generalized Birkhoff's result into various directions by considering more general (weighted) composition operators $C_{w,\psi}(f):=w\cdot (f\circ\psi)$ for $f\in\mathscr{H}(X)$, where $w\in\mathscr{H}(X)$ and $\psi:X\rightarrow X$ is holomorphic (see e.g.\ \cite{BernalMontes}, \cite{Grosse-ErdmannMortini}, \cite{YousefiRezaei}, \cite{Bes}, and the references therein). For holomorphic functions in several variables some partial results on hypercyclicity for special unweighted composition operators have been obtained in \cite{AbeZappa}, \cite{Bernal}, and \cite{Leon-Saavedra}, and only recently for arbitrary composition operators (even for holomorphic functions in several variables on a connected Stein manifold) in \cite{Zajac}.

Considering harmonic functions instead of holomorphic functions, in \cite{Dzagnidze}, \cite{Armitage} hypercyclicity of generalized translation operators has been investigated in this context while sufficient conditions for hypercyclicity of such special composition operators on kernels of elliptic partial differential operators have been presented in \cite{CalderonMueller}, among other things. A sufficient condition for hypercyclicity of general composition operators on arbitrary kernels of partial differential operators was obtained in \cite{KalmesNiess}.

While for unweighted composition operators on the space of real analytic functions $\mathscr{A}(X)$ on open subsets $X$ of $\R^d$ topological transitivity has been investigated in \cite{BonetDomanski}, hypercyclicity and topological mixing of weighted composition operators on the space of smooth functions $C^\infty(X)$, where again $X\subseteq\R^d$ is open, have only very recently been characterized in \cite{Przestacki}. 

For Banach spaces of functions dynamical properties of (weighted) composition operators have also been investigated. There are many results in the context of Banach spaces of holomorphic functions, see e.g.\ \cite{BourdonShapiro}, \cite{GallardoMontes}, \cite{MirallesWolf} and references therein. Characterizations of hypercyclicity of weighted composition operators on Banach spaces of continuous functions and on $L^p(\mu)$-spaces have been obtained in \cite{Kalmes}, among other things.\\

The aim of the present paper is to present a general approach to topological transitivity and topological (weak) mixing of weighted composition operators on locally convex spaces of scalar valued functions which are defined by local properties. In section \ref{local function spaces} we introduce the setting of general locally convex sheaves of functions which gives the appropriate general framework for our objective. This general approach enables us to give an almost characterization of these properties in section \ref{dynamical properties} (Theorem \ref{almost characterization of transitivity} and Theorem \ref{almost characterization of mixing}). For many concrete function spaces these almost characterizations can be made into characterizations with only a minimal additional effort. This is shown in section \ref{dynamics of weighted composition operators on concrete local function spaces} where we not only recover and unify most of the above mentioned results by our general approach but where our abstract setting also permits to improve these results into one direction or the other.

Moreover, as one of the main application of our approach we show in section \ref{ultradifferentiable} that topological transitivity and topological weak mixing for weighted composition operators $C_{w,\psi}$ on spaces of ultradifferentiable functions of Beurling type as well as of Roumieu type, respectively, are equivalent and we characterize these properties together with topological mixing in terms of the weight $w$ and the symbol $\psi$ (Theorems \ref{dynamics in Beurling} and \ref{dynamics in Roumieu}).

As a second main application of our abstract results we show in section \ref{kernels of elliptic differential operators} that for weighted composition operators on kernels of elliptic partial differential operators in $C^\infty(X)$ for open subsets $X$ of $\R^d$ which are homeomorphic to $\R^d$, hypercyclicity and topological weak mixing are equivalent whenever $|w|\leq 1$, and we again characterize these properties as well as topological mixing in terms of the weight function and the symbol of the operator (Theorem \ref{dynamics on elliptic kernels}). We pay special attention to eigenspaces of the Laplace operator (Corollary \ref{dynamics on eigenspaces of the laplace operator}) and of the Cauchy-Riemann operator (Corollary \ref{dynamics on eigenspaces of the cauchy-riemann operator}). In the latter case we show that weak mixing, hypercyclicity, and mixing for weighted composition operators coincide whenever $|w|\leq 1$. Additionally, we characterize in terms of the weight and symbol which weighted composition operators are well-defined on the eigenspaces of the Laplace operator and of the Cauchy-Riemann operator (Proposition \ref{well-definedness on eigenspaces of laplace and cauchy-riemann}).\\ 

Throughout, we use standard notation and terminology from functional analysis. For anything related to functional analysis which is not explained in the text we refer the reader to \cite{MeiseVogt}. Moreover, we use common notation from the theory of linear partial differential operators. For this we refer the reader to \cite{Hoermander}. Finally, for notions and results from dynamics of linear operators which are not explained in the text we refer the reader to \cite{BayartMatheron2} and \cite{Grosse-ErdmannPeris}. 

\section{Function spaces defined by local properties}\label{local function spaces}

In order to deal with weighted composition operators on several function spaces at once we choose the framework of sheaves. Here and in the sequel let $\K\in\{\R,\C\}$.

\begin{definition}
	Let $\Omega$ be a locally compact, $\sigma$-compact, non-compact Hausdorff space and $\mathscr{F}$ a sheaf of functions on $\Omega$, i.e.\
	\begin{itemize}
		\item For each open subset $X\subseteq\Omega$ $\mathscr{F}(X)$ is a vector space of $\K$-valued functions and if $Y\subseteq\Omega$ is another open set with $Y\subseteq X$ the restriction mapping
		$$r_X^Y:\mathscr{F}(X)\rightarrow\mathscr{F}(Y),f\mapsto f_{|Y}$$
		is well-defined.
		
		\item (Localization) For every open cover $(X_\iota)_{\iota\in I}$ of an open set $X\subseteq\Omega$ and for each $f,g\in\mathscr{F}(X)$ with $f_{|X_\iota}=g_{|X_\iota} (\iota\in I)$ it holds $f=g$.
		
		\item (Gluing) For each open cover $(X_\iota)_{\iota\in I}$ of an open set $X\subseteq\Omega$ and for all $(f_\iota)_\iota\in\prod_{\iota\in I}\mathscr{F}(X_\iota)$ with $f_{\iota|X_\iota\cap X_\kappa}=f_{\kappa|X_\iota\cap X_\kappa}\, (\iota,\kappa\in I)$ there is $f\in\mathscr{F}(X)$ with $f_{|X_\iota}=f_\iota\,(\iota\in I)$.
	\end{itemize}
	From these defining properties of a sheaf of functions it follows immediately that for every open subset $X\subseteq\Omega$ and each open, relatively compact exhaustion $(X_n)_{n\in\N_0}$ of $X$ the spaces $\mathscr{F}(X)$ and the projective limit $\mbox{proj}_{\leftarrow n}(\mathscr{F}(X_n),r_{X_{n+1}}^{X_n})$, i.e.\ the subspace
	$$\{(f_n)_{n\in\N_0}\in\prod_{n\in\N_0}\mathscr{F}(X_n);\,\forall\,n\in\N_0: f_n=r_{X_{n+1}}^{X_n}(f_{n+1})\}$$
	of $\prod_{n\in\N_0}\mathscr{F}(X_n)$ are algebraically isomorphic via the mapping
	$$\mathscr{F}(X)\rightarrow\mbox{proj}_{\leftarrow n}(\mathscr{F}(X_n),r_{X_{n+1}}^{X_n}),f\mapsto (r_X^{X_n}(f))_{n\in\N_0}=(f_{|X_n})_{n\in\N_0},$$
	where injectivity follows from the localization property and surjectivity from the gluing property of a sheaf.
	
	In order to be able to apply results from functional analysis, we define the following properties for a sheaf of functions $\mathscr{F}$ on $\Omega$:
	
	\begin{enumerate}
		\item[($\mathscr{F}1$)] For every open subset $X\subseteq\Omega$ the function space $\mathscr{F}(X)$ is a webbed and ultrabornological Hausdorff locally convex space (which is satisfied, for example, if $\mathscr{F}(X)$ is a Fr\'echet space). Additionally, we assume that $\mathscr{F}(X)\subseteq C(X)$ for every open $X\subseteq\Omega$ and that for each $x\in X$ the point evaluation $\delta_x$ in $x$ is a continuous linear functional on $\mathscr{F}(X)$.
		
		It thus follows that for open $X,Y\subseteq\Omega$ with $Y\subseteq X$ the restriction map $r_X^Y$ has closed graph, hence is continuous by De Wilde's Closed Graph Theorem (see e.g.\ \cite[Theorem 24.31]{MeiseVogt}).
		
		Moreover, for every open, relatively compact exhaustion $(X_n)_{n\in\N_0}$ of $X$ we assume that the above mentioned algebraic isomorphism between $\mathscr{F}(X)$ and $\mbox{proj}_{\leftarrow n}(\mathscr{F}(X_n),r_{X_{n+1}}^{X_n})$ is even a topological isomorphism.
		
		\item[($\mathscr{F}2$)] For every compact $K\subseteq \Omega$ there is $f_K\in\mathscr{F}(\Omega)$ such that $f_K(x)\neq 0$ for each $x\in K$.
		
		\item[($\mathscr{F}3$)] For every pair of distinct points $x,y$ in $\Omega$ there is  $f\in\mathscr{F}(\Omega)$ with $f(x)=0$ and $f(y)=1$.
	\end{enumerate}
\end{definition}

\newpage

\begin{remark}\label{properties of local function sheaves}
	\hspace{2em}
	\begin{itemize}
		\item[i)] For a sheaf $\mathscr{F}$ on $\Omega$ with the property that $\mathscr{F}(X)$ is a locally convex space and $r_X^Y$ is continuous for each open $Y\subseteq X\subseteq\Omega$ (by definition, this means that $\mathscr{F}$ is a locally convex sheaf on $\Omega$) it follows immediately that for any open, relatively compact exhaustion $(X_n)_{n\in\N}$ of $X$ the canonical isomorphism between $\mathscr{F}(X)$ and $\mbox{proj}_{\leftarrow n}(\mathscr{F}(X_n),r_{X_{n+1}}^{X_n})$ is continuous. Therefore, if $\mathscr{F}$ is a locally convex sheaf of continuous functions such that $\mathscr{F}(X)$ is a Fr\'echet space for each open $X\subseteq\Omega$ on which $\delta_x$ is a continuous linear functional for every $x\in X$, it follows from the Open Mapping Theorem and the fact that Fr\'echet spaces are ultrabornological (see e.g.\ \cite[Remark 24.15 c)]{MeiseVogt}) and webbed (see e.g.\ \cite[Corollary 24.29]{MeiseVogt}) that $(\mathscr{F}1)$ is satisfied. 
		\item[ii)] For a sheaf $\mathscr{F}$ on $\Omega$ satisfying $(\mathscr{F}1)$ it follows from $\delta_x\in\mathscr{F}(X)'$ for each $x\in X$ that the inclusion mapping
		$$\mathscr{F}(X)\hookrightarrow C(X),f\mapsto f$$
		has closed graph - where we equip $C(X)$ as usual with the compact-open topology. Since $\mathscr{F}(X)$ is supposed to be ultrabornological it follows from De Wilde's Closed Graph Theorem that this inclusion is continuous, i.e.\ the topology carried by $\mathscr{F}(X)$ is finer than the compact-open topology.
		\item[iii)] If $\mathscr{F}$ contains constant functions property $(\mathscr{F}3)$ means precisely that $\mathscr{F}(\Omega)$ separates points.
	\end{itemize}
\end{remark}

\begin{example}\label{examples of sheaves}
	\hspace{2em}
	\begin{itemize}
		\item[i)] For a $\sigma$-compact, locally compact, non-compact Hausdorff space $\Omega$ the sheaf $C$ of continuous functions satisfies $(\mathscr{F}1)-(\mathscr{F}3)$, i.e.\ for an open subset $X\subseteq\Omega$ let $C(X)$ be the space of $\K$-valued continuous functions on $X$ equipped with the compact-open topology, that is, the locally convex topology defined by the family of seminorms $\{\|\cdot\|_K;\,K\subset X\mbox{ compact}\}$, where for a compact subset $K\subset X$
		$$\forall\,f\in C(X):\,\|f\|_K:=\sup_{x\in K}|f(x)|.$$
		Recall that locally compact spaces are completely regular (see e.g.\ \cite[Theorem 3.3.1]{Engelking}) so that $(\mathscr{F}3)$ is indeed satisfied.
		
		\item[ii)] For $\Omega=\R^d$ and $r\in\N_0\cup\{\infty\}$ we denote by $C^r$ the sheaf of $r$-times continuously differentiable functions, i.e.\ for every open $X\subseteq\R^d$ let $C^r(X)$ be the space of $\K$-valued functions which are $r$-times continuously differentiable. We equip $C^r(X)$ with the topology of local uniform convergence of all partial derivatives up to order $r$, i.e.\ the locally convex topology defined by the family of seminorms $\{\|\cdot\|_{l,K};\,l\in\N_0, l<r+1, K\subset X\mbox{ compact}\}$, where for $l<r+1$ and $K\subset X$ compact
		$$\forall\,f\in C^r(X):\|f\|_{l,K}:=\sup_{|\alpha|\leq l}\sup_{x\in K}|\partial^\alpha f(x)|.$$
		This makes $C^r(X)$ a separable Fr\'echet space and the sheaf $C^r$ on $\R^d$ is easily seen to satisfy $(\mathscr{F}1)-(\mathscr{F}3)$.
		
		\item[iii)] For $\Omega=\C$ let $\mathscr{H}$ be the sheaf of holomorphic functions, i.e.\ for $X\subseteq\C$ let $\mathscr{H}(X)$ denote the space of holomorphic functions on $X$ endowed with the compact-open topology. Then $\mathscr{H}(X)$ is a separable Fr\'echet space and it follows easily that $(\mathscr{F}1)-(\mathscr{F}3)$ are satisfied. More generally than this example is v).
		
		\item[iv)] Let again $\Omega=\R^d$ and denote by $\mathscr{A}$ the sheaf of real analytic functions, that is, for open $X\subseteq\R^d$ $\mathscr{A}(X)$ is the space of real analytic functions of $X$. One apparent way of equipping $\mathscr{A}(X)$ with a locally convex topology is by considering the finest locally convex topology such that all the restriction maps
		$$\mathscr{H}(U)\rightarrow\mathscr{A}(X),f\mapsto f_{|X}$$
		are continuous, where $U\subseteq\C^d$ is an arbitrary open set for which $X\subseteq U\cap\R^d$ and where $\mathscr{H}(U)$ is equipped with the compact-open topology. It is not hard to see that this (inductive) topology is Hausdorff. Another way of endowing $\mathscr{A}(X)$ with a natural locally convex topology is by taking the initial topology with respect to all restriction maps
		$$r_K:\mathscr{A}(X)\rightarrow\mathscr{H}(K), f\mapsto f_{|K},$$
		where $K\subseteq X$ is an arbitrary compact set and $\mathscr{H}(K)$ denotes the space of germs of holomorphic functions on $K$ equipped with the locally convex inductive limit topology $\mathscr{H}(K)=\mbox{ind}_{m\in\N}\mathscr{H}^\infty(U_m)$, where $(U_m)_{m\in\N}$ is a (decreasing) basis of $\C^d$-neighborhoods of $K$ and $\mathscr{H}^\infty(U_m)$ denotes the Banach space of bounded holomorphic functions on $U_m$ equipped with the supremum norm. Choosing a compact exhaustion $(K_n)_{n\in\N}$ of $X$ it follows that topologized in this way $\mathscr{A}(X)$ equals the (topological) projective limit of the projective sequence $(\mathscr{H}(K_n))_{n\in\N}$ with restrictions as linking maps. 
		
		Again, it is not hard to see that the first (inductive) topology on $\mathscr{A}(X)$ is coarser than the second (projective) topology. That these two topologies actually coincide is a fundamental result due to Martineau (see \cite{Martineau}). Since $\mathscr{H}(U)$ is a Fr\'echet space, thus ultrabornological, it follows that $\mathscr{A}(X)$ is ultrabornological as the inductive limit of ultrabornological spaces. As an $LB$-space, $\mathscr{H}(K)$ is webbed for every compact $K\subseteq X$ and therefore $\mathscr{A}(X)$ is webbed as the countable projective limit of webbed spaces.
		
		Now let $X\subseteq\R^d$ be open and let $(X_n)_{n\in\N_0}$ be an open, relatively compact exhaustion of $X$
		$$i:\mathscr{A}(X)\rightarrow\mbox{proj}_{\leftarrow n}(\mathscr{A}(X_n),r_{X_{n+1}}^{X_n}), f\mapsto (r_X^{X_n}(f))_{n\in\N_0}.$$
		We show that the continuous bijection $i$ is open. By the projective description of the topology on $\mathscr{A}(X)$ for any zero neighborhood $V$ in $\mathscr{A}(X)$ there is a compact subset $K$ of $X$ and an absolutely convex zero neighborhood $W$ in $\mathscr{H}(K)$ such that $r_K^{-1}(W)\subseteq V$. Thus, for every complex neighborhood $U$ of $K$ there is $\delta>0$ with
		$$r_K^{-1}\big(\{g\in\mathscr{H}(K);\,g\in\mathscr{H}^\infty(U),\|g\|_{\infty,U}<\delta\}\big)\subseteq V,$$
		where $\|\cdot\|_{\infty,U}$ denotes the supremum norm over $U$. If $m\in\N_0$ is chosen such that $K\subseteq X_m$ it follows with
		$$\rho_K:\mathscr{A}(X_m)\rightarrow\mathscr{H}(K),f\mapsto f_{|K}$$
		that 
		$$V_m:=\rho_K^{-1}\big(\{g\in\mathscr{H}(K);\,g\in\mathscr{H}^\infty(U),\|g\|_{\infty,U}<\delta\}\big)$$
		is a zero neighborhood in $\mathscr{A}(X_m)$. Clearly,
		\begin{align*}
			V_m\cap r_X^{X_m}(\mathscr{A}(X))&=r_X^{X_m}\Big(r_K^{-1}\big(\{g\in\mathscr{H}(K);\,g\in\mathscr{H}^\infty(U),\|g\|_{\infty,U}<\delta\}\big)\Big)\\
			&\subseteq r_X^{X_m}(V).	
		\end{align*}
		If we set for $k\in\N_0$
		$$\pi_k:\mbox{proj}_{\leftarrow n}(\mathscr{A}(X_n),r_{X_{n+1}}^{X_n})\rightarrow \mathscr{A}(X_k), (f_n)_{n\in\N_0}\mapsto f_k$$
		it follows
		$$i(V)=\pi_m^{-1}\big(r_X^{X_m}(V)\big)\supseteq\pi_m^{-1}\big(V_m\cap r_X^{X_m}(\mathscr{A}(X))\big)=\pi_m^{-1}(V_m)$$
		so that $i(V)$ is a zero neighborhood in $\mbox{proj}_{\leftarrow n}(\mathscr{A}(X_n),r_{X_{n+1}}^{X_n})$ which proves that $i$ is open. Thus, the sheaf $\mathscr{A}$ on $\R^d$ satisfies $(\mathscr{F}1)$. Moreover, $(\mathscr{F}2)$ as well as $(\mathscr{F}3)$ are obviously satisfied, too.
		
		\item[v)] For $\Omega=\R^d$ and a complex coefficient polynomial in $d$ variables $P$, i.e.\ $P\in\C[X_1,\ldots,X_d]$ we define for an open $X\subseteq\R^d$
		$$C^\infty_P(X):=\{f\in C^\infty(X);\, P(\partial)f=0 \mbox{ in }X\},$$
		where as usual for $P(\xi)=\sum_{|\alpha|\leq m}a_\alpha\xi^\alpha$ we set
		$$\forall\,f\in C^\infty(X), x\in X:\,P(\partial)f(x)=\sum_{|\alpha|\leq m}a_\alpha\partial^\alpha f(x).$$
		Obviously, $C_P^\infty(X)$ is a subspace of $C^\infty(X)$ and since $P(\partial)$ is a continuous linear operator on the separable Fr\'echet space $C^\infty(X)$ it follows that $C_P^\infty(X)$ is a closed subspace of $C^\infty(X)$ thus a separable Fr\'echet space itself when equipped with the relative topology of $C^\infty(X)$. It is easily seen that $C_P^\infty$ is a sheaf on $\R^d$ which satisfies $(\mathscr{F}1)$. In order to see that $(\mathscr{F}2)$ holds as well, let $\zeta\in\C^d$ satisfy $P(\zeta)=0$ - we exclude the rather boring case of a constant polynomial $P$. Then
		$$e_\zeta:\R^d\rightarrow\C, e_\zeta(x)=\exp(\sum_{j=1}^d \zeta_j x_j)$$
		belongs to $C_P^\infty(\R^d)$ which shows that $(\mathscr{F}2)$ is indeed satisfied.
		
		However, $(\mathscr{F}3)$ need not be satisfied for general (non-constant) $P$ as is seen by taking $d=2$ and $P(\xi_1,\xi_2)=\xi_1$. Then $C_P^\infty(X)$ consists obviously of functions which are independ on $x_1$ and there is no $f\in C_P^\infty(\R^2)$ with $f(0,0)=1$ and $f(1,0)=0$. However, by Proposition \ref{f3 and f4 for kernel} $(\mathscr{F}3)$ is in particular satisfied by $C_P^\infty$ whenever $P$ is (hypo)elliptic.
		
		Considering the special cases of $P(\partial)$ being the Cauchy-Riemann operator, resp.\ the Laplace operator, gives as $C_P^\infty$ the sheaf of holomorphic functions $\mathscr{H}$ on open subsets of $\C$, resp.\ the sheaf of harmonic functions on open subsets of $\R^d$.
	\end{itemize}
\end{example}

In most of the above examples the considered sheaves are sheaves of $C^1$-functions over open subsets of euclidean space. For these kind of sheaves we introduce yet another property.

\begin{definition}
Let $\Omega=\R^d$ and let $\mathscr{F}$ be a sheaf satisfying $(\mathscr{F}1)$. Then we define:
\begin{itemize}
	\item[$(\mathscr{F}4)$] For every open $X\subseteq\R^d$ we have $\mathscr{F}(X)\subseteq C^1(X)$ and for every $x\in X$, $1\leq j\leq d$ the distribution of order one
	$$-\partial_j\delta_x:\mathscr{F}(X)\rightarrow\K,f\mapsto \partial_j f(x)$$
	is continuous. Moreover, for each $h\in\R^d\backslash\{0\},\lambda\in\K$, and $x\in X$ the kernel of the continuous linear functional
	$$\mathscr{F}(X)\rightarrow\K,f\mapsto\sum_{j=1}^dh_j\partial_j f(x)-\lambda f(x)$$
	is not all of $\mathscr{F}(X)$. 
\end{itemize}
\end{definition}

Property $(\mathscr{F}4)$ implies that for every direction $h\in\R^d\backslash\{0\}$ and any $x\in X$ the directional derivative in direction $h$ evaluated at $x$ does not coincide with a multiple of $\delta_x$. Clearly, $(\mathscr{F}4)$ is satisfied by Examples \ref{examples of sheaves} ii) for $r\geq 1$, iii), and iv). For hypoelliptic polynomials $P$ the sheaf $C_P^\infty$ satisfies $(\mathscr{F}4)$ by Proposition \ref{f3 and f4 for kernel}.\\

\textbf{General assumption.}
Let $\mathscr{F}$ be a sheaf on $\Omega$ satisfying $(\mathscr{F}1)$, $X\subseteq \Omega$ open, and let $w:X\rightarrow\K$ as well as $\psi:X\rightarrow X$ be continuous. $w$ will be called a \textit{weight} and $\psi$ a \textit{symbol}. We assume that the weighted composition operator
$$C_{w,\psi}:\mathscr{F}(X)\rightarrow\mathscr{F}(X), f\mapsto w\cdot(f\circ\psi)$$
is well-defined. For every $x\in X$ we have by hypothesis that $\delta_x\in\mathscr{F}(X)'$ and it follows easily from the Hahn-Banach Theorem, that the linear span of $\{\delta_x;\,x\in X\}$ is weak*-dense in $\mathscr{F}(X)'$. Since $\mathscr{F}(X)$ is Hausdorff, it follows herefrom that $C_{w,\psi}$ has closed graph. By $(\mathscr{F}1)$ we conclude from De Wilde's Closed Graph Theorem \cite[Theorem 24.31]{MeiseVogt} that $C_{w,\psi}$ is continuous. 

\section{Dynamical properties of weighted composition operators}\label{dynamical properties}

In this section we give necessary and sufficient conditions for a weighted composition operator defined on a local space of functions to be topologically transitive and topologically (weakly) mixing, respectively. We will see that in many concrete cases these necessary and sufficient conditions coincide thus they yield a characterization of said properties.

\begin{definition}
Let $E$ be a locally convex space and $T$ a continuous linear operator on $E$.
\begin{itemize}
	\item[i)] $T$ is called \emph{(topologically) transitive} if for any pair of non-empty, open subsets $U,V$ of $E$ there is $m\in\N$ such that $T^m(U)\cap V\neq\emptyset$.
	\item[ii)] $T$ is called \emph{(topologically) weakly mixing} if $T\oplus T$ is transitive on $E\oplus E$, i.e.\ if for every choice of non-empty, open subsets $U_j,V_j$ of $E, j=1,2,$ there is $m\in\N$ such that $T^m(U_j)\cap V_j\neq\emptyset$.
	\item[iii)] $T$ is called \emph{(topologically) mixing} if for every pair $U,V$ of non-empty, open subsets of $E$ there is $M\in\N$ such that $T^m(U)\cap V\neq\emptyset$ for every $m\geq M$.
	\item[iv)] $T$ is called \emph{hypercyclic} if there is $x\in E$ whose orbit under $T$, i.e.\ $\mbox{orb}(T,x):=\{T^mx;\,m\in\N_0\}$ is dense in $E$.
\end{itemize}
\end{definition}

\begin{remark}
\begin{itemize}
	\item[i)] Clearly, $E$ has to be separable in order to support a hypercyclic operator. By Birkhoff's Transitivity Criterion \cite[Theorem 2.19]{Grosse-ErdmannPeris} a continuous linear operator on a separable Fr\'echet space is transitive if and only if it is hypercyclic.
	\item[ii)] Obviously, every mixing operator is weakly mixing and every weakly mixing operator is transitive. In general, the reverse implications are not true. While it is not too complicated to give an example of a weakly mixing operator which is not mixing (see e.g.\ \cite[Remark 4.10]{Grosse-ErdmannPeris}) it is highly intricate to construct a hypercyclic operator (on a Banach space) which fails to be weakly mixing. Such an operator was constructed by De la Rosa and Read \cite{DelaRosaRead} (see also \cite{BayartMatheron}) who thereby solved a problem posed by Herrero \cite{Herrero} which remained open for more than fifteen years.
	\item[iii)] It follows immediately from the definition that every transitive operator has dense image.
\end{itemize}
\end{remark}

We first give necessary conditions for a weighted composition operator defined on a local space of functions to be transitive. This result is inspired by \cite[Lemma 3.1]{Przestacki}. Before we state the condition we recall a definition for symbols.

\begin{definition}\label{run-away}
	Let $X\subseteq\Omega$ be open and $\psi:X\rightarrow X$ be continuous.
	\begin{itemize}
		\item[i)] $\psi$ is called \emph{run-away} if for each compact subset $K$ of $X$ there is $m\in\N$ such that $\psi^m(K)\cap K=\emptyset$, where $\psi^m$ denotes the $m$-fold iterate of $\psi$, i.e.\ $\psi^m=\psi\circ\ldots\circ\psi$ with $m$ composition factors.
		\item[ii)] $\psi$ is called \emph{strong run-away} if for each compact subset $K$ of $X$ there is $M\in\N$ such that $\psi^m(K)\cap K=\emptyset$ whenever $m\geq M$.
	\end{itemize}
\end{definition}

Clearly, strong run-away implies run-away. The converse is true for continuous and injective $\psi:\R\rightarrow\R$, as has been shown in \cite[Lemma 4.1]{Przestacki}. Moreover, for a holomorphic $\psi:X\rightarrow X$ on a simply connected domain $X\subsetneq \C$ run-away and strong run-away are the same, due to the Riemann Mapping Theorem combined with the Denjoy-Wolff Iteration Theorem (see e.g.\ \cite[Chapter 5]{Shapiro}) and in fact these properties are equivalent to $\psi$ having no fixed point. To the author's best knowledge, there is no example of a continuous $\psi:X\rightarrow X$ on an open subset $X$ of a  locally compact, $\sigma$-compact, non-compact Hausdorff space $\Omega$ which is run-away but not strong run-away. 

\begin{proposition}\label{necessary transitive}
	Let $\mathscr{F}$ be a sheaf on $\Omega$ satisfying $(\mathscr{F}1)-(\mathscr{F}3)$. Let $X\subseteq\Omega$ be non-empty and open and assume that the weighted composition operator $C_{w,\psi}$ is transitive on $\mathscr{F}(X)$. Then the following hold.
	\begin{itemize}
		\item[i)] $\forall\,x\in X:\, w(x)\neq 0$.
		\item[ii)] $\psi$ is injective.
		\item[iii)] $\psi$ has no fixed points.
		\item[iv)] $\forall\,x\in X:\,\overline{\{\psi^m(x);\,m\in\N_0\}}$ is not a compact subset of $X$, where the closure is taken in $X$.
		\item[v)] $\psi$ is run-away under any of the following additional assumptions:
		\begin{itemize}
			\item[v-1)] $\forall\,x\in X:\,|w(x)|\leq 1$.
			\item[v-2)] $C_{w,\psi}$ is weakly mixing.
			\item[v-3)] $\mathscr{F}(\Omega)$ is dense in $C(\Omega)$ with respect to the compact-open topology.
		\end{itemize}
	\end{itemize}
	Additionally, if $\mathscr{F}$ satisfies $(\mathscr{F}4)$ then we also have
	\begin{itemize}
		\item[vi)] $\forall\,x\in X:\,\mbox{det}J\psi(x)\neq 0$, where $J\psi(x)$ denotes the Jacobian of $\psi$ in $x$.
	\end{itemize}
\end{proposition}

\begin{proof}
	In order to prove i), assume that $w(x_0)=0$ for some $x_0\in X$. Then the image of $C_{w,\psi}$ is contained in the kernel of $\delta_{x_0}$ which is a closed subspace of $\mathscr{F}(X)$ due to $(\mathscr{F}1)$. By $(\mathscr{F}3)$ this closed subspace is proper, in particular nowhere dense. Hence, the image of $C_{w,\psi}$ is not dense in $\mathscr{F}(X)$ contradicting that $C_{w,\psi}$ is transitive, so that i) follows.
	
	Next, we assume that $\psi(x)=\psi(y)$ for $x,y\in X, x\neq y$. Because $w(y)\neq 0$ by i) it follows
	$$\mbox{im }C_{w,\psi}\subseteq\mbox{kern }(\delta_x-\frac{w(x)}{w(y)}\delta_y).$$
	Since $w(x)\neq 0$ by i) it follows from $(\mathscr{F}1)$ and $(\mathscr{F}3)$ that $\mbox{kern }(\delta_x-\frac{w(x)}{w(y)}\delta_y)$ is a closed, proper subspace of $\mathscr{F}(X)$. Since the image of $C_{w,\psi}$ is dense in $\mathscr{F}(X)$ we obtain a contradiction as in the proof of i), so that ii) follows.
	
	To prove iii) we denote for $\alpha\in\K$ and $r>0$ the open ball about $\alpha$ with radius $r$ as $B(\alpha,r)$ and the corresponding closed ball by $B[\alpha,r]$. Assume there is $x_0\in X$ with $\psi(x_0)=x_0$. In case of $|w(x_0)|\leq 1$ we have
	$$\forall\,f\in\delta_{x_0}^{-1}(B(0,1)), n\in\N_0:\,|C^n_{w,\psi}(f)(x_0)|\leq 1$$
	and thus
	\begin{equation}\label{contradiction1}
	\forall\,n\in\N_0:C_{w,\psi}^n(\delta_{x_0}^{-1}(B(0,1)))\cap \delta_{x_0}^{-1}(B(2,1))=\emptyset.
	\end{equation}
	But since $\mbox{kern }\delta_{x_0}\neq\mathscr{F}(X)$ by $(\mathscr{F}3)$ we have $\mbox{im }\delta_{x_0}=\K$ so $\delta_{x_0}^{-1}(B(2,1))$ is a non-empty, open (by $(\mathscr{F}1)$) subset of $\mathscr{F}(X)$, as is $\delta_{x_0}^{-1}(B(0,1))$. Hence, (\ref{contradiction1}) contradicts the transitivity of $C_{w,\psi}$.
	In case of $|w(x_0)|>1$ we have
	$$\forall\,f\in\delta_{x_0}^{-1}(\K\backslash B[0,1]), n\in\N:\,|C^n_{w,\psi}(f)(x_0)|>1$$
	implying
	$$\forall\,n\in\N_0:\,C_{w,\psi}^n(\delta_{x_0}^{-1}(\K\backslash B[0,1]))\cap\delta_{x_0}^{-1}(B(0,1))=\emptyset$$
	which again contradicts the transitivity of $C_{w,\psi}$. Thus, iii) is proved.
	
	In order to prove iv), assume that for some $x_0\in X$ the set
	$$K:= \overline{\{\psi^m(x_0);\,m\in\N_0\}}$$
	is compact. By i), there are $a,b>0$ such that
	$$\forall\,x\in K:\,a\leq |w(x)|\leq b.$$
	Fix $f_K$ according to $(\mathscr{F}2)$ and set
	$$\alpha:=\inf_{x\in K}|f_K(x)|\mbox{ and }\beta:=\sup_{x\in K}|f_K(x)|.$$
	Then $\alpha>0$ and $\beta<\infty$ and the set
	$$U:=\{f\in\mathscr{F}(X);\,\forall\,x\in K:\,\frac{\alpha}{2}<|f(x)|<2\beta\}$$
	obviously contains $f_K$ and is open with respect to the compact-open topology. Thus, $U$ is an open neighborhood of $f_K$ in $\mathscr{F}(X)$. A straightforward calculation gives
	\begin{equation}\label{inequality1}
	\forall\,f\in U, m\in\N_0:|\frac{C_{w,\psi}^m(f)(\psi(x_0))}{C_{w,\psi}^m(f)(x_0)}|\leq\frac{4b\beta}{a\alpha}.
	\end{equation}
	By $(\mathscr{F}1)$ it follows that
	$$V:=\{f\in\mathscr{F}(X);\,|\delta_{\psi(x_0)}(f)|>\frac{4b\beta}{a\alpha}|\delta_{x_0}(f)|\}$$
	is an open subset of $\mathscr{F}(X)$ and since $\psi(x_0)\neq x_0$ by iii) it follows from $(\mathscr{F}3)$ that $V\neq \emptyset$. From (\ref{inequality1}) we obtain
	$$\forall\,m\in\N_0:\,C_{w,\psi}^m(U)\cap V=\emptyset$$
	which contradicts the transitivity of $C_{w,\psi}$. Thus, iv) is proved.
	
	We continue with the proof of v) and argue again by contradiction. Assume there is a compact subset $K$ of $X$ with
	$$\forall\,m\in \N\,\exists\,x_m\in K:\,\psi^m(x_m)\in K.$$
	Due to $(\mathscr{F}2)$ there is $f_K\in\mathscr{F}(X)$ with $f_K(x)\neq 0$ for every $x\in K$. We set
	$$\alpha:=\inf_{y\in K}|f_K(y)|>0, \beta:=\sup_{y\in K}|f_K(y)|<\infty.$$
	
	We first assume that additionally $|w(x)|\leq 1$ holds for every $x\in X$. We define
	$$U_1:=\{g\in\mathscr{F}(X);\,\forall\,x\in K:\,|g(x)|<\frac{\alpha}{2}\}$$
	and
	$$V_1:=\{g\in\mathscr{F}(X);\,\forall\,x\in K:\,|f_K(x)-g(x)|<\frac{\alpha}{2}\}$$
	which are open with respect to the compact-open topology and therefore, by $(\mathscr{F}1)$, open subsets of $\mathscr{F}(X)$. Obviously, $f_K\in V_1$ and $0\in U_1$. For every $m\in\N$ and each $g\in U_1$ it follows from $|w|\leq 1$ and $x_m,\psi^m(x_m)\in K$
	\begin{align*}
		\max_{x\in K}|f_K(x)-C_{w,\psi}^m(g)(x)|&\geq |f_K(x_m)-\prod_{j=0}^{m-1}w(\psi^j(x_m))g(\psi^m(x_m))|\\
		&\geq \alpha-|g(\psi^m(x_m))|>\frac{\alpha}{2}
	\end{align*} 
	so that $C_{w,\psi}^m(U_1)\cap V_1=\emptyset$ which gives the desired contradiction to the transitivity of $C_{w,\psi}$.
	
	Next we assume that $C_{w,\psi}$ is even weakly mixing. We define
	$$U_2:=\{g\in\mathscr{F}(X);\,\forall\,x\in K:\,\frac{\alpha}{2}<|g(x)|<2\beta\}$$
	and
	$$V_2:=\{g\in\mathscr{F}(X);\,\sup_{x\in K}|g(x)|<\frac{\alpha^2}{8\beta}\}$$
	which are open with respect to the compact-open topology and thus open subsets of $\mathscr{F}(X)$ by $(\mathscr{F}1)$. Because $0\in V_2$ and $f_K\in U_2$ and since $C_{w,\psi}$ is weakly mixing there is $m\in\N$ such that
	$$C_{w,\psi}^m(U_2)\cap U_2\neq \emptyset\mbox{ and }C_{w,\psi}^m(U_2)\cap V_2\neq\emptyset.$$
	Pick $f\in U_2$ with $C_{w,\psi}^m(f)\in U_2$. By $\psi^m(x_m)\in K$ we have $|f(\psi^m(x_m))|<2\beta$ so that because of $x_m\in K$
	\begin{align}\label{repairing inequality 1}
		\frac{\alpha}{2}&<|C_{w,\psi}^m(f)(x_m)|=|\prod_{j=0}^{m-1}w(\psi^j(x_m))f(\psi^m(x_m))|\\
		&\leq |\prod_{j=0}^{m-1}w(\psi^j(x_m))|2\beta.\nonumber
	\end{align}
	On the other hand, with $g\in U_2$ such that $C_{w,\psi}^m(g)\in V_2$ it follows from $\psi^m(x_m)\in K$ and thus $|g(\psi^m(x_m))|>\alpha/2$ with $x_m\in K$
	\begin{align*}
		\frac{\alpha^2}{8\beta}>|C_{w,\psi}^mg(x_m)|=|\prod_{j=0}^{m-1}w(\psi^j(x_m))g(\psi^m(x_m))|>\frac{\alpha}{2}|\prod_{j=0}^{m-1}w(\psi^j(x_m))|
	\end{align*}
	which contradicts (\ref{repairing inequality 1}).
	
	In order to finish the proof of v), we next assume that $\mathscr{F}(\Omega)$ is dense in $C(\Omega)$. Because of i) there are $a,b>0$ such that
	$$\forall\,x\in K:\,a\leq|w(x)|\leq b.$$
	By $(\mathscr{F}2)$ there is $\tilde{f}_K\in\mathscr{F}(X)$ such that $\tilde{f}_K(x)\neq 0$ for every $x\in K\cup\psi(K)$ so that
	$$\tilde{\alpha}:=\inf_{x\in K\cup\psi(K)}|\tilde{f}_K(x)|>0\mbox{ and }\tilde{\beta}:=\sup_{x\in K\cup\psi(K)}|\tilde{f}_K(x)|<\infty.$$
	We define
	$$U_3:=\{g\in\mathscr{F}(X);\,\forall\,x\in K\cup\psi(K):\,\frac{\tilde{\alpha}}{2}<|g(x)|<2\tilde{\beta}\}$$
	which is an open neighborhood of $\tilde{f}_K$ in $\mathscr{F}(X)$ satisfying
	$$\forall\,g\in U_3, m\in \N_0:\,|\frac{C_{w,\psi}^m(g)(\psi(x_m))}{C_{w,\psi}^m(g)(x_m)}|\leq\frac{4b\tilde{\beta}}{a\tilde{\alpha}},$$
	that is
	\begin{equation}\label{contradiction2-1}
	\forall\,g\in U_3, m\in \N_0:\,|\delta_{\psi(x_m)}\big(C_{w,\psi}^m(g)\big)|\leq\frac{4b\tilde{\beta}}{a\tilde{\alpha}}|\delta_{x_m}\big(C_{w,\psi}^m(g)\big)|.
	\end{equation}
	In particular
	$$\forall\,m\in\N_0:\,C_{w,\psi}^m(U_3)\cap (V_3\cap\mathscr{F}(X))=\emptyset,$$
	where
	$$V_3:=\{g\in C(X);\,\forall\,x\in K:\,|\delta_{\psi(x)}(g)|>\frac{4b\tilde{\beta}}{a\tilde{\alpha}}|\delta_x(g)|\}.$$
	Once we have shown that $V_3\cap\mathscr{F}(X)$ is a non-empty open subset of $\mathscr{F}(X)$ this will yield the desired contradiction to the transitivity of $C_{w,\psi}$.
	
	A straightforward calculation shows that for $g,h\in C(X)$ and $x\in K$ we have
	$$|\delta_{\psi(x)}(h)|-\frac{4b\tilde{\beta}}{a\tilde{\alpha}}|\delta_x(h)|\geq|\delta_{\psi(x)}(g)|-\frac{4b\tilde{\beta}}{a\tilde{\alpha}}|\delta_x(g)|-\frac{8b\tilde{\beta}}{a\tilde{\alpha}}\sup_{y\in K\cup\psi(K)}|g(y)-h(y)|$$
	which shows that $V_3$ is an open subset of $C(X)$ with respect to the compact-open topology, thus $V_3\cap \mathscr{F}(X)$ is open in $\mathscr{F}(X)$ by $(\mathscr{F}1)$. Since $\mathscr{F}(\Omega)$ is dense in $C(\Omega)$ and $\{f_{|X};f\in C(\Omega)\}$ is dense in $C(X)$ by an application of Urysohn's Lemma \cite[Theorem 1.5.11]{Engelking}, it is enough to show that $V_3$ is not empty in order to prove $V_3\cap \mathscr{F}(X)\neq\emptyset$.
	
	In order to show that $V_3\neq\emptyset$ we use a clever construction from \cite[Lemma 3.2]{Przestacki}. Let $Y\subseteq X$ be an open, relatively compact neighborhood of $K$. By Urysohn's Lemma there is $h\in C(X)$ such that $h_{|K}=0$, $h_{|X\backslash Y}=1$ and $0\leq h\leq 1$. The series
	$$f:=\sum_{m=0}^\infty\big(\frac{a\alpha}{4b\beta+a\alpha}\big)^m h\circ\psi^m$$
	converges uniformly on $X$, in particular $f\in C(X)$ and because of iv) it follows that $f(x)>0$ for all $x\in X$. Moreover, for each $x\in K$ we have
	$$|\delta_{\psi(x)}(f)|=f(\psi(x))=\frac{4b\beta+a\alpha}{a\alpha}f(x)>\frac{4b\beta}{a\alpha}|\delta_x(f)|,$$
	i.e.\ $f\in V_3$. Thus, v) is proved.
	
	Finally, let $\mathscr{F}$ satisfy $(\mathscr{F}4)$ in addition to $(\mathscr{F}1)-(\mathscr{F}3)$. Assuming the existence of $x_0\in X$ with $\mbox{det}J\psi(x_0)=0$ there is $h\in\R^d\backslash\{0\}$ in the kernel of $J\psi(x_0)$. By an easy calculation we have
	$$\forall\,f\in\mathscr{F}(X):\,\langle\nabla C_{w,\psi}(f)(x_0), h\rangle=C_{w,\psi}(f)(x_0)\frac{1}{w(x_0)}\langle\nabla w(x_0),h\rangle,$$
	showing
	$$\mbox{im }C_{w,\psi}\subseteq\mbox{kern }\big(f\mapsto\langle\nabla f(x_0),h\rangle-\frac{1}{w(x_0)}\langle\nabla w(x_0),h\rangle\delta_{x_0}(f)\big).$$
	By $(\mathscr{F}4)$ the previously mentioned kernel is a closed, proper subspace of $\mathscr{F}(X)$ which contradicts that the image of $C_{w,\psi}$ is dense in $C_{w,\psi}$.
\end{proof}

\begin{definition}\label{local action}
	Let $\mathscr{F}$ be a sheaf on $\Omega$ satisfying $(\mathscr{F}1)$ and $X\subseteq\Omega$ be open such that the weighted composition operator $C_{w,\psi}$ is well-defined on $\mathscr{F}(X)$. $C_{w,\psi}$ is said to \emph{act locally on} $\mathscr{F}(X)$ if for every open subset $Y$ of $X$
	$$C_{w,\psi,Y}:\mathscr{F}(Y)\rightarrow\mathscr{F}(\psi^{-1}(Y)), f\mapsto w\cdot(f\circ\psi),$$
	i.e.\ the weighted composition operator $C_{w,\psi}$ (formally) applied to functions defined only on $Y$, is well-defined.
\end{definition}

\begin{remark}
	Clearly, under hypothesis $(\mathscr{F}1)$, for every open subset $Y$ of $\Omega$ and any $f\in\mathscr{F}(Y)$ the function
	$$\psi^{-1}(Y)\rightarrow\K,y\mapsto w(y) f(\psi(y))$$
	is a well-defined continuous function.
	
	If $C_{w,\psi}$ operates locally on $\mathscr{F}(X)$ it follows from $(\mathscr{F}1)$ and De Wildes's Closed Graph Theorem that $C_{w,\psi,Y}$ is continuous from $\mathscr{F}(Y)$ to $\mathscr{F}(\psi^{-1}(Y))$ for any open subset $Y$ of $X$. It follows immediately from the definition of the restriction maps $r_X^Y$ etc.\ and $C_{w,\psi,Y}$ that
	$$C_{w,\psi,Y}\circ r_X^Y=r_X^{\psi^{-1}(Y)}\circ C_{w,\psi},$$
	or more generally
	\begin{equation}\label{commuting relation}
		\forall\,m\in\N:\,C_{w,\psi,(\psi^{m-1})^{-1}(Y)}\circ\ldots\circ C_{w,\psi,Y}\circ r_X^Y=r_X^{(\psi^m)^{-1}(Y)}\circ C_{w,\psi}^m
	\end{equation}
		for every open $Y\subseteq X$.
\end{remark}

\begin{proposition}\label{sufficiency of run-away}
Let $\mathscr{F}$ be a sheaf on $\Omega$ satisfying $(\mathscr{F}1)$, $X\subseteq\Omega$ be open such that $C_{w,\psi}$ acts locally on $\mathscr{F}(X)$. Let $(X_n)_{n\in\N}$ be a relatively compact, open exhaustion of $X$. Assume that the following conditions are satisfied
\begin{itemize}
	\item[a)] For every open, relatively compact subset $Y$ of $X$ and any $m\in\N_0$
	$$r_X^{(\psi^m)^{-1}(Y)}(\mathscr{F}(X))\subseteq\overline{(C_{w,\psi,(\psi^{m-1})^{-1}(Y)}\circ\ldots\circ C_{w,\psi,Y})(\mathscr{F}(Y))},$$
	where the closure is taken in $\mathscr{F}((\psi^m)^{-1}(Y))$.
	\item[b)] There are $m,n\in\N$ with the properties
	\begin{itemize}
		\item[b1)] $\psi^m(X_n)$ is an open subset of $X$.
		\item[b2)] $X_n\cap\psi^m(X_n)=\emptyset$.
		\item[b3)] The restriction map $r_X^{X_n\cup\psi^m(X_n)}$ has dense range.
	\end{itemize} 
\end{itemize}
Then, for every $f,g\in\mathscr{F}(X)$ and any absolutely convex zero neighborhood $U_n$ in $\mathscr{F}(X_n)$ we have
$$\emptyset\neq C_{w,\psi}^m\big(f+(r_X^{X_n})^{-1}(U_n)\big)\cap \big(g+(r_X^{X_n})^{-1}(U_n)\big).$$
\end{proposition}

\begin{proof}
	Fix $f,g\in\mathscr{F}(X)$ and an absolutely convex zero neighborhood $U_n$ in $\mathscr{F}(X_n)$. By b1), $Y:=\psi^m(X_n)$ is an open, relatively compact subset of $X$ and $Z:=(\psi^m)^{-1}(Y)$ is an open subset of $X$ with $X_n\subseteq Z$ so that $(r_Z^{X_n})^{-1}(\frac{1}{2} U_n)$ is a zero neighborhood in $\mathscr{F}(Z)$. By hypothesis a), there is $\tilde{g}\in\mathscr{F}(Y)$ such that with
	$$T:=C_{w,\psi,(\psi^{m-1})^{-1}(Y)}\circ\ldots\circ C_{w,\psi,Y}$$
	(note that $T$ is continuous) we have
	$$T(\tilde{g})-r_X^Z(g)\in(r_Z^{X_n})^{-1}(\frac{1}{2} U_n).$$
	Because $X_n\cap\psi^m(X_n)=X_n\cap Y=\emptyset$ it follows from the gluing property of a sheaf that there is $f_n\in\mathscr{F}(X_n\cup Y)$ with
	$$r_{X_n\cup Y}^{X_n}(f_n)=r_X^{X_n}(f)\mbox{ and }r_{X_n\cup Y}^Y(f_n)=\tilde{g}.$$
	By hypothesis b3) there is $h\in\mathscr{F}(X)$ such that
	$$r_X^{X_n\cup Y}(h)-f_n\in (r_{X_n\cup Y}^{X_n})^{-1}(\frac{1}{2}U_n)\cap\big(r_Z^{X_n}\circ T\circ r_{X_n\cup Y}^Y\big)^{-1}(\frac{1}{2}U_n).$$
	Thus
	$$\frac{1}{2}U_n\ni r_{X_n\cup Y}^{X_n}\big(r_X^{X_n\cup Y}(h)-f_n\big)=r_X^{X_n}(h)-r_{X_n\cup Y}^{X_n}(f_n)=r_X^{X_n}(h)-r_X^{X_n}(f),$$
	so that
	$$h\in f+(r_X^{X_n})^{-1}(\frac{1}{2}U_n).$$
	We will show that $C_{w,\psi}^m(h)\in g+(r_X^{X_n})^{-1}(U_n)$. Indeed, using (\ref{commuting relation}) we have
	\begin{align*}
		(r_X^{X_n}\circ C_{w,\psi}^m)(h)-r_X^{X_n}(g)&=(r_Z^{X_n}\circ r_X^Z\circ C_{w,\psi}^m)(h)\\
		&-(r_Z^{X_n}\circ T)(\tilde{g})+(r_Z^{X_n}\circ T)(\tilde{g})-r_Z^{X_n}\circ r_X^Z(g)\\
		&=(r_Z^{X_n}\circ r_X^{(\psi^m)^{-1}(Y)}\circ C_{w,\psi}^m)(h)\\
		&-(r_Z^{X_n}\circ T\circ r_{X_n\cup Y}^Y)(f_n)+r_Z^{X_n}(T(\tilde{g})-r_X^Z(g))\\
		&=(r_Z^{X_n}\circ C_{w,\psi,(\psi^{m-1})^{-1}(Y)}\circ\ldots\circ C_{w,\psi,Y}\circ r_{X_n\cup Y}^Y\\
		&\circ r_X^{X_n\cup Y})(h)-(r_Z^{X_n}\circ T\circ r_{X_n\cup Y}^Y)(f_n)\\
		&+r_Z^{X_n}(T(\tilde{g})-r_X^Z(g))\\
		&=(r_Z^{X_n}\circ T\circ r_{X_n\cup Y}^Y)(r_X^{X_n\cup Y}(h)-f_n)\\
		&+r_Z^{X_n}(T(\tilde{g})-r_X^Z(g))\\
		&\in\frac{1}{2}U_n+\frac{1}{2}U_n\subseteq U_n,
	\end{align*}
	so that $C_{w,\psi}^m(h)\in g+(r_X^{X_n})^{-1}(U_n)$. Thus we have shown
	$$C_{w,\psi}^m\big(f+(r_X^{X_n})^{-1}(U_n)\big)\cap\big(g+(r_X^{X_n})^{-1}(U_n)\big)\neq\emptyset.$$
\end{proof}

That condition a) from the previous proposition is in particular satisfied if $C_{w,\psi}$ has dense range is the content of the next one.

\begin{proposition}\label{dense range}
	Let $\mathscr{F}$ be a sheaf on $\Omega$ satisfying $(\mathscr{F}1)$, $X\subseteq\Omega$ be open such that $C_{w,\psi}$ acts locally on $\mathscr{F}(X)$. Assume that $C_{w,\psi}$ has dense range. Then, for every open subset $Y$ of $X$ and any $m\in\N_0$
	$$r_X^{(\psi^m)^{-1}(Y)}(\mathscr{F}(X))\subseteq\overline{(C_{w,\psi,(\psi^{m-1})^{-1}(Y)}\circ\ldots\circ C_{w,\psi,Y})(\mathscr{F}(Y))},$$
	where the closure is taken in $\mathscr{F}((\psi^m)^{-1}(Y))$.
\end{proposition}

\begin{proof}
	Fix an open subset $Y$ of $X$ and $m\in\N$. From the hypothesis on the range of $C_{w,\psi}$, the continuity of the restriction map $r_X^{(\psi^m)^{-1}(Y)}$, and the commutativity relation (\ref{commuting relation}) it follows
	\begin{align*}
		r_X^{(\psi^m)^{-1}(Y)}(\mathscr{F}(X))&=r_X^{(\psi^m)^{-1}(Y)}(\overline{C_{w,\psi}^m(\mathscr{F}(X)})\subseteq\overline{r_X^{(\psi^m)^{-1}(Y)}(C_{w,\psi}^m(\mathscr{F}(X)))}\\
		&=\overline{(C_{w,\psi,(\psi^{m-1})^{-1}(Y)}\circ\ldots\circ C_{w,\psi,Y})(\mathscr{F}(Y))}.
	\end{align*}
\end{proof}

We now come to an almost characterization of weak mixing for weighted composition operators acting locally on $\mathscr{F}(X)$, where $\mathscr{F}$ is a sheaf of functions defined by local properties which satisfies $(\mathscr{F}1)-(\mathscr{F}3)$.

\begin{theorem}\label{almost characterization of transitivity}
Let $\mathscr{F}$ be a sheaf on $\Omega$ satisfying $(\mathscr{F}1)-(\mathscr{F}3)$, let $X\subseteq\Omega$ be open, and assume that the weighted composition operator $C_{w,\psi}$ acts locally on $\mathscr{F}(X)$. Then, among the following, $i)\Rightarrow ii)\Rightarrow iv)$. Additionally, if $\mathscr{F}(\Omega)$ is dense in $C(\Omega)$ with respect to the compact-open topology or if $|w(x)|\leq 1$ for all $x\in X$, it holds $i)\Rightarrow ii)\Rightarrow iii)\Rightarrow iv)$.
\begin{itemize}
	\item[i)] a) For any $m\in\N_0$ and every open, relatively compact $Y\subseteq \psi^m(X)$
	$$r_X^{(\psi^m)^{-1}(Y)}(\mathscr{F}(X))\subseteq\overline{\big(C_{w,\psi,(\psi^{m-1})^{-1}(Y)}\circ\ldots\circ C_{w,\psi,Y}\big)(\mathscr{F}(Y))},$$
	where the closure is taken in $\mathscr{F}\big((\psi^m)^{-1}(Y)\big)$.
	
	\noindent b) There is an open, relatively compact exhaustion $(X_n)_{n\in\N_0}$ of $X$ such that for every $n\in\N_0$ there is $m\in\N$ with
	\begin{itemize}
		\item[b1)] $\psi^m(X_n)$ is open.
		\item[b2)] $X_n\cap \psi^m(X_n)=\emptyset$.
		\item[b3)] The restriction map $r_X^{X_n\cup\psi^m(X_n)}$ has dense range.
	\end{itemize}
	\item[ii)] $C_{w,\psi}$ is weakly mixing on $\mathscr{F}(X)$.
	\item[iii)] $C_{w,\psi}$ is transitive on $\mathscr{F}(X)$.
	\item[iv)] a) from i) holds, $w$ has no zeros, $\psi$ is injective and run-away, and in case of $(\mathscr{F}4)$ with continuously differentiable $w$ and $\psi$, additionally $\mbox{det}J\psi(x)\neq 0$ for every $x\in X$.
\end{itemize}
\end{theorem}

\begin{remark}\label{considering the almost characterization}
	\hspace{2em}
	\begin{itemize}
		\item[i)] If iv) of the above theorem is valid, $\psi$ is in particular run-away so that b2) is satisfied for any open, relatively compact exhaustion $(X_n)_{n\in\N_0}$ of $X$ for suitable $m$.
		\item[ii)] In case of $\Omega=\R^d$, if $\psi$ is injective it follows from Brouwer's Invariance of Domain Theorem (see \cite[Corollary 19.9]{Bredon} for an even stronger result) that for any open subset $Y$ of $X$ and each $m\in\N_0$ the set $\psi^m(Y)$ is open. Thus, in case of $\Omega=\R^d$, the only obstruction against the equivalence of i), ii), and iv) in Theorem \ref{almost characterization of transitivity} is the existence of a particular open, relatively compact exhaustion $(X_n)_{n\in\N}$ which satisfies b3). In concrete situations, this obstruction is overcome by a suitable approximation result which - depending on the concrete sheaf of functions under consideration - can be trivial (e.g.\ in the case of continuous functions) or highly sophisticated (as in the case of holomorphic functions in several variables \cite{Zajac}).
		\item[iii)] Although looking possibly rather deterrering condition a) from item i) in Theorem \ref{almost characterization of transitivity} is in most concrete situations fulfilled for zero-free $w$ and injective as well as open (the latter is redundant for $\Omega=\R^d$ by Brouwer's Invariance of Domain Theorem) $\psi$ because of the following. Under the stated hypothesis on $w$ and $\psi$ for any open subset $Y$ of $\psi^m(X), m\in\N_0,$ and each $f\in\mathscr{F}((\psi^m)^{-1}(Y))$ the function
		$$\tilde{f}:Y\mapsto\K,y\mapsto\big(\frac{f}{\prod_{j=0}^{m-1}w(\psi^j(\cdot))}\big)\circ(\psi^m)^{-1}(y)$$
		is a well-defined, continuous function. In case of $\tilde{f}\in\mathscr{F}(Y)$ a straightforward calculation gives
		$$\big(C_{w,\psi,(\psi^{m-1})^{-1}(Y)}\circ\ldots\circ C_{w,\psi,Y}\big)(\tilde{f})=f$$
		in $(\psi^m)^{-1}(Y)$. But in many concrete examples, $\tilde{f}\in\mathscr{F}(Y)$ indeed holds - if also $\mbox{det}J\psi(x)\neq 0$ for all $x\in X$. In case of $\mathscr{F}=C^r$ it follows from the fact that if $C_{w,\psi}$ is well-defined then $w\in C^r(X)$ and therefore $\frac{1}{w}\in C^r(X)$, $\psi$ is $C^r$, too, and by the Inverse Function Theorem (see e.g.\ \cite[Theorem 1.3.2, Remark 1.3.11]{Narasimhan}) the same holds for $(\psi^m)^{-1}$ so that $\tilde{f}\in C^r(Y)$. The same arguments also hold for real analytic or holomorphic functions in several variables (see e.g.\ again \cite{Narasimhan}, resp.\ \cite[Theorem 7.5]{FritzscheGrauert}). Thus, in many concrete examples, condition a) from i) of Theorem \ref{almost characterization of transitivity} is redundant in part iv).
	\end{itemize}
\end{remark}

\begin{proof}[Proof of Theorem \ref{almost characterization of transitivity}]
	Assume that i) holds and let $V_j,W_j\subseteq\mathscr{F}(X)$ be non-empty and open, $j=1,2$. We fix $f_j\in V_j, g_j\in W_j, j=1,2$. Since by $(\mathscr{F}1)$ we have $\mathscr{F}(X)=\mbox{proj}_{\leftarrow n}(\mathscr{F}(X_n),r_{X_{n+1}}^{X_n})$ topologically, there is $n\in\N$ and an absolute convex zero neighborhood $U_n$ in $\mathscr{F}(X_n)$ such that for $j=1,2$ we have
	$$f_j+(r_X^{X_n})^{-1}(U_n)\subseteq V_j\mbox{ and }g_j+(r_X^{X_n})^{-1}(U_n)\subseteq W_j$$
	(see e.g.\ \cite[Chapter 3.3]{Wengenroth}). From hypotheses a) and b) it follows with the aid of Proposition \ref{sufficiency of run-away} that there is $m\in\N$ such that for $j=1,2$ we have
	$$\emptyset\neq C_{w,\psi}^m\big(f_j+(r_X^{X_n})^{-1}(U_n)\big)\cap(g_j+(r_X^{X_n})^{-1}(U_n))\subseteq C_{w,\psi}^m(V_j)\cap W_j,$$
	so that ii) holds.
	
	ii) obviously implies iii).
	
	If ii) is satisfied (respectively if iii) is satisfied and $\mathscr{F}(\Omega)$ is dense in $C(\Omega)$ or $|w|\leq 1$) $C_{w,\psi}$ has in particular dense range so that condition a) in i) follows from Proposition \ref{dense range} while the rest of the properties listed in iv) follow from Proposition \ref{necessary transitive}.
\end{proof}

Concrete applications of Theorem \ref{almost characterization of transitivity} will be postponed to sections \ref{dynamics of weighted composition operators on concrete local function spaces}, \ref{ultradifferentiable}, and \ref{kernels of elliptic differential operators}. We next come to an almost characterization of mixing of weighted composition operators on local function spaces.

\begin{theorem}\label{almost characterization of mixing}
	Let $\mathscr{F}$ be a sheaf on $\Omega$ satisfying $(\mathscr{F}1)-(\mathscr{F}3)$, let $X\subseteq\Omega$ be open, and assume that the weighted composition operator $C_{w,\psi}$ acts locally on $\mathscr{F}(X)$. Then, among the following, $i)\Rightarrow ii)\Rightarrow iii)$:
	\begin{itemize}
		\item[i)] a) For any $m\in\N_0$ and every open, relatively compact $Y\subseteq \psi^m(X)$
		$$r_X^{(\psi^m)^{-1}(Y)}(\mathscr{F}(X))\subseteq\overline{\big(C_{w,\psi,(\psi^{m-1})^{-1}(Y)}\circ\ldots\circ C_{w,\psi,Y}\big)(\mathscr{F}(Y))},$$
		where the closure is taken in $\mathscr{F}\big((\psi^m)^{-1}(Y)\big)$.
		
		\noindent b) There is an open, relatively compact exhaustion $(X_n)_{n\in\N_0}$ of $X$ such that for every $n\in\N_0$ there is $M\in\N$ with
		\begin{itemize}
			\item[b1)] $\psi^m(X_n)$ is open for all $m\geq M$.
			\item[b2)] $X_n\cap \psi^m(X_n)=\emptyset$ for all $m\geq M$.
			\item[b3)] The restriction map $r_X^{X_n\cup\psi^m(X_n)}$ has dense range for all $m\geq M$.
		\end{itemize}
		\item[ii)] $C_{w,\psi}$ is mixing on $\mathscr{F}(X)$.
		\item[iii)] a) from i) holds, $w$ has no zeros, $\psi$ is injective and strong run-away, and in case of $(\mathscr{F}4)$ with continuously differentiable $w$ and $\psi$, additionally $\mbox{det}J\psi(x)\neq 0$ for every $x\in X$.
	\end{itemize}
\end{theorem}  

\begin{proof}
	In order to show that i) implies ii) let $V,W\subseteq\mathscr{F}(X)$ be open and non-empty. As in the proof of the implication "i)$\Rightarrow$ ii)" of Theorem \ref{almost characterization of transitivity}, let $f\in V, g\in W$, and choose $n\in\N$ and an absolutely convex zero neighborhood $U_n$ in $\mathscr{F}(X_n)$ such that
	$$f+(r_X^{X_n})^{-1}(U_n)\subseteq V\mbox{ and }g+(r_X^{X_n})^{-1}(U_n)\subseteq W.$$
	From the hypotheses a), b1)-b3) it follows together with Proposition \ref{sufficiency of run-away} that there is $M\in\N$ such that for any $m\geq M$
	$$\emptyset\neq C_{w,\psi}^m\big(f+(r_X^{X_n})^{-1}(U_n)\big)\cap (g+(r_X^{X_n})^{-1}(U_n))\subseteq C_{w,\psi}^m(V)\cap W,$$
	so that $C_{w,\psi}$ is mixing.
	
	If on the other hand $C_{w,\psi}$ is mixing, $C_{w,\psi}$ is in particular weakly mixing so that by Theorem \ref{almost characterization of transitivity} we only have to show that $\psi$ is strong run away. Assume that this is not the case, i.e.\ that there is a compact subset $K$ of $X$ and a strictly increasing sequence of natural numbers $(m_l)_{l\in\N}$ such that
	$$\forall\,l\in\N\,\exists\, x_l\in K:\, \psi^{m_l}(x_l)\in K.$$
	By $(\mathscr{F}2)$ there is $f_K\in\mathscr{F}(X)$ such that $f_K(x)\neq 0$ for all $x$ in $K$. Then
	$$\alpha:=\inf_{x\in K}|f_K(x)|>0\mbox{ and }\beta:=\sup_{x\in K}|f_K(x)|<\infty$$
	and the set
	$$U:=\{g\in\mathscr{F}(X);\,\forall\,x\in K:\,\frac{\alpha}{2}<|g(x)|<2\beta\}$$
	contains $f_K$ and is open with respect to the compact-open topology and is therefore an open neighborhood of $f_K$ in $\mathscr{F}(X)$.
	The set
	$$V:=\{g\in\mathscr{F}(X);\,\sup_{x\in K}|g(x)|<\frac{\alpha^2}{8\beta}\}$$
	is open in $\mathscr{F}(X)$, too, and contains the zero function. Since $C_{w,\psi}$ is mixing, there is $M\in\N$ such that
	$$\forall\,m\geq M:\,C_{w,\psi}^m(U)\cap U\neq \emptyset\mbox{ and }C_{w,\psi}^m(U)\cap V\neq \emptyset.$$
	Now we fix $l\in\N$ with $m_l>M$ and pick $f\in U$ with $C_{w,\psi}^{m_l}(f)\in U$ as well as $g\in U$ with $C_{w,\psi}^{m_l}(g)\in V$. As in the proof of Proposition \ref{necessary transitive} v) under the additional assumption v-2) one deduces the contradiction
	$$\frac{\alpha}{4\beta}<|\prod_{j=0}^{m_l-1} w(\psi^j(x_l))|<\frac{\alpha}{4\beta}.$$
	Thus, $\psi$ is strong run-away.
\end{proof}

\section{Dynamics of weighted composition operators on concrete local function spaces}\label{dynamics of weighted composition operators on concrete local function spaces}

As a first application of the results from the previous section we next show how to use them to recover characterizations of transitivity/hypercyc\-li\-ci\-ty and mixing of weighted composition operators on various concrete functions spaces obtained by several authors or which we assume to be well-known  (and add a slight generalization here and there).\\

\textbf{4.1.\ Continuous functions.} For an arbitrary locally compact, $\sigma$-compact, non-compact Hausdorff space $\Omega$ the sheaf of continuous functions $C$ satisfies properties $(\mathscr{F}1)-(\mathscr{F}3)$ as explained in example \ref{examples of sheaves} i). Clearly, for an arbitrary open subset $X\subseteq\Omega$, $w\in C(X)$, and continuous $\psi:X\rightarrow X$ the weighted composition operator $C_{w,\psi}$ is well-defined on $C(X)$ and acts locally on $C(X)$.

Recall that locally compact spaces are completely regular and that by \cite[Theorem 5]{Warner} for a completely regular topological space $Z$ the space $C(Z)$ equipped with the compact-open topology is separable, if and only if $Z$ has a separable metrizable compression, i.e.\ if and only if $Z$ has a weaker separable metrizable topology. Thus, in case the open subset $X$ of $\Omega$ below has a weaker separable metrizable topology, part a) of the next result characterizes hypercyclicity of $C_{w,\psi}$ on $C(X)$.

\begin{corollary}\label{topological space}
	Let $\Omega$ be a locally compact, $\sigma$-compact, non-compact Hausdorff space, $X\subseteq \Omega$ be open, $w\in C(X)$ and $\psi:X\rightarrow X$ be continuous. If $\Omega\neq\R^d$ we assume additionally that $\psi$ is open.
	\begin{itemize}
		\item[a)] The following are equivalent.
		\begin{itemize}
			\item[i)] $C_{w,\psi}$ is weakly mixing on $C(X)$.
			\item[ii)] $C_{w,\psi}$ is transitive on $C(X)$.
			\item[iii)] $w$ has no zeros, $\psi$ is injective and run-away.
		\end{itemize}
		\item[b)] The following are equivalent.
		\begin{itemize}
			\item[i)] $C_{w,\psi}$ is mixing on $C(X)$.
			\item[ii)] $w$ has no zeros, $\psi$ is injective and strong run-away.
		\end{itemize}
	\end{itemize}
\end{corollary}

\begin{proof}
	We first prove a). Clearly, i) implies ii) and by Theorem \ref{almost characterization of transitivity}, ii) implies iii). Now, if iii) is satisfied it follows from Remark \ref{considering the almost characterization} that condition i) a) from Theorem \ref{almost characterization of transitivity} is fulfilled. Let $(X_n)_{n\in\N_0}$ be an arbitrary open, relatively compact exhaustion of $X$. If $\Omega=\R^d$ it follows from Brouwer's Invariance of Domain Theorem, $\psi^m$ is an open mapping on $X$ for any $m$, in particular, for every $n\in\N_0$ and any $m\in \N$, $\psi^m(X_n)$ is an open subset of $X$. In case of $\Omega\neq\R^d$ the same follows from the hypotheses on $\psi$. For fixed $n\in\N_0$, since $\psi$ is run-away, there is $m\in\N$ such that
	$$X_n\cap\psi^m(X_n)=\emptyset.$$
	Let $K$ be a compact subset of $X_n\cup\psi^m(X_n)$. Since $X$ is a locally compact Hausdorff space, $X$ is normal and therefore in particular $X$ is regular. Thus, for every $x\in K$ there is an open neighborhood $V_x$ of $x$ such that $\overline{V_x}\subseteq X_n\cup\psi^m(X_n)$. Using the compactness of $K$ we deduce from this that there is an open neighborhood $V$ of $K$ such that $\overline{V}\subseteq X_n\cup\psi^m(X_n)$. By Urysohn's Lemma \cite[Theorem 1.5.11]{Engelking} there is $h\in C(X)$ such that $h=1$ on $K$ and $h=0$ on $X\backslash V$, thus $\supp h\subseteq \overline{V}$. Now if $g\in C(X_n\cup\psi^m(X_n))$ we obtain a continuous function $f$ on $X$ with $f_{|K}=g$ by extending $h g$ by zero outside of $X_n\cup\psi^m(X_n)$. Since $K\subset X_n\cup\psi^m(X_n)$ was an arbitary compact set it follows that $r_X^{X_n\cup\psi^m(X_n)}$ has dense range. Hence, condition i) b) from Theorem \ref{almost characterization of transitivity} is also satisfied, so that iii) implies i).
	
	Referring to Theorem \ref{almost characterization of mixing} instead of Theorem \ref{almost characterization of transitivity} the proof of b) is mutatis mutandis the same as the proof of part a).
\end{proof}

\textbf{4.2.\ $C^r$-functions on open subsets of $\R^d$.} Let $\mathscr{F}$ be the sheaf $C^r$ of $r$-times continuously differentiable functions on $\R^d$ (equipped with the topology of local uniform convergence of partial derivatives of order less than $r+1$). Then $C^r$ satisfies properties $(\mathscr{F}1)-(\mathscr{F}4)$ as explained in example \ref{examples of sheaves} ii) and $C^r(X)$ is a separable Fr\'echet space for every open subset $X\subseteq\R^d$. Clearly, for an arbitrary open subset $X\subseteq\R^d$, $w\in C^r(X)$, and a $C^r$-function $\psi:X\rightarrow X$ the weighted composition operator $C_{w,\psi}$ is well-defined on $C^r(X)$ and acts locally on $C^r(X)$. The next application of the results from the previous section gives the results obtained by Przestacki in \cite{Przestacki} in case of $r=\infty$.

\begin{corollary}\label{smooth case}
	Let $X\subseteq \R^d$ be open, $r\in\N\cup\{\infty\}$, $w\in C^r(X)$ and $\psi:X\rightarrow X$ be a $C^r$-function.
	\begin{itemize}
		\item[a)] The following are equivalent.
		\begin{itemize}
			\item[i)] $C_{w,\psi}$ is weakly mixing on $C^r(X)$.
			\item[ii)] $C_{w,\psi}$ is hypercyclic on $C^r(X)$.
			\item[iii)] $w$ has no zeros, $\psi$ is injective, run-away, and $\mbox{det}J\psi(x)\neq 0$ for all $x\in X$.
		\end{itemize}
		\item[b)] The following are equivalent.
		\begin{itemize}
			\item[i)] $C_{w,\psi}$ is mixing on $C^r(X)$.
			\item[ii)] $w$ has no zeros, $\psi$ is injective, strong run-away, and $\mbox{det}J\psi(x)\neq 0$ for all $x\in X$.
		\end{itemize}
	\end{itemize}
\end{corollary}

\begin{proof}
	We first prove a). Since $C^r(X)$ is a separable Fr\'echet space, by Birkhoff's transitivity criterion, hypercyclicity of $C_{w,\psi}$ is equivalent to transitivity. Thus, i) implies ii) and since polynomials, a fortiori $C^r(\R^d)$, are dense in $C(\R^d)$ (see e.g.\ \cite[Chapter 15, Corollary 4]{Treves}) by Theorem \ref{almost characterization of transitivity}, ii) implies iii). Now, if iii) is satisfied it follows from Remark \ref{considering the almost characterization} that condition i) a) from Theorem \ref{almost characterization of transitivity} is fulfilled. Let $(X_n)_{n\in\N_0}$ be an arbitrary open, relatively compact exhaustion of $X$. Because $\mbox{det}J(x)\neq 0$ for all $x\in X$ it follows together with the injectivity of $\psi$ that $\psi^m$ is an open mapping (see e.g.\ \cite[Theorem 1.3.2]{Narasimhan}). In particular, for every $n\in\N_0$ and any $m\in \N$, $\psi^m(X_n)$ is an open subset of $X$. For fixed $n\in\N_0$, since $\psi$ is run-away, there is $m\in\N$ such that
	$$X_n\cap\psi^m(X_n)=\emptyset.$$
	Let $K$ be a compact subset of $X_n\cup\psi^m(X_n)$ and $V$ an open $K$ such that $\overline{V}\subseteq X_n\cup\psi^m(X_n)$. Then there is $h\in C^\infty(X)$ such that $h=1$ on $K$ and $\supp h\subseteq \overline{V}$. Now if $g\in C^\infty(X_n\cup\psi^m(X_n))$ we obtain a smooth function $f$ on $X$ with $f_{|K}=g$ by extending $h g$ by zero outside of $X_n\cup\psi^m(X_n)$. Since $K\subseteq X_n\cup\psi^m(X_n)$ was an arbitary compact set we conclude that $r_X^{X_n\cup\psi^m(X_n)}$ has dense range so that i) b) from Theorem \ref{almost characterization of transitivity} is fulfilled. Thus, iii) implies i).
	
	Referring again to Theorem \ref{almost characterization of mixing} instead of Theorem \ref{almost characterization of transitivity} the proof of b) is mutatis mutandis the same as the proof of part a).
\end{proof}

\textbf{4.3. Real analytic functions on open subsets of $\R^d$.} Let $\mathscr{F}$ be the sheaf $\mathscr{A}$ of real analytic functions on $\R^d$ (equipped with its natural topology, see Example \ref{examples of sheaves} v)). Then $\mathscr{A}$ satisfies properties $(\mathscr{F}1)-(\mathscr{F}4)$ as explained in Examples \ref{examples of sheaves} v) and clearly, for an arbitrary open subset $X\subseteq\R^d$, $w\in \mathscr{A}(X)$, and real analytic $\psi:X\rightarrow X$ the weighted composition operator $C_{w,\psi}$ is well-defined on $\mathscr{A}(X)$ and acts locally on $\mathscr{A}(X)$. For the special case of $w=1$, the equivalence of ii) and iii) in part a) of our next application of the results from the previous section was obtained by Bonet and Doma\'nski in \cite[Theorem 2.3]{BonetDomanski}.

\begin{corollary}\label{real analytic case}
	Let $X\subseteq \R^d$ be open, $w\in \mathscr{A}(X)$ and $\psi:X\rightarrow X$ be real analytic.
	\begin{itemize}
		\item[a)] The following are equivalent.
		\begin{itemize}
			\item[i)] $C_{w,\psi}$ is weakly mixing on $\mathscr{A}(X)$.
			\item[ii)] $C_{w,\psi}$ is transitive on $\mathscr{A}(X)$.
			\item[iii)] $w$ has no zeros, $\psi$ is injective, run-away, and $\mbox{det}J\psi(x)\neq 0$ for all $x\in X$.
		\end{itemize}
		\item[b)] The following are equivalent.
		\begin{itemize}
			\item[i)] $C_{w,\psi}$ is mixing on $\mathscr{A}(X)$.
			\item[ii)] $w$ has no zeros, $\psi$ is injective, strong run-away, and $\mbox{det}J\psi(x)\neq 0$ for all $x\in X$.
		\end{itemize}
	\end{itemize}
\end{corollary}

\begin{proof}
	As before, we first prove a). Obviously, i) implies ii) and because polynomials are dense in $C(\R^d)$ (see e.g.\ \cite[Chapter 15, Corollary 4]{Treves}) by Theorem \ref{almost characterization of transitivity}, ii) implies iii). Now, if iii) is satisfied it follows from Remark \ref{considering the almost characterization} that condition i) a) from Theorem \ref{almost characterization of transitivity} is fulfilled. Let $(X_n)_{n\in\N_0}$ be an arbitrary open, relatively compact exhaustion of $X$. As in the proof of Corollary \ref{smooth case} it follows that for every $n\in\N_0$ and any $m\in \N$, $\psi^m(X_n)$ is an open subset of $X$. For fixed $n\in\N_0$, since $\psi$ is run-away, there is $m\in\N$ such that
	$$X_n\cap\psi^m(X_n)=\emptyset,$$
	so that b1) and b2) from i) in Theorem \ref{almost characterization of transitivity} are fulfilled.
	
	In order to show that b3) is fulfilled, too, let $f\in\mathscr{A}(X_n\cup\psi^m(X_n))$ be arbitrary and let $V$ be any neighborhood of $f$ in $\mathscr{A}(X_n\cup\psi^m(X_n))$. By the definition of the topology on $\mathscr{A}(X_n\cup\psi^m(X_n))$ there is a compact subset $K$ of $X_n\cup \psi^m(X_n)$ and a complex neighborhood $W_0\subseteq\C^d$ of $K$ such that $f$ extends to a holomorphic function on $W_0$ and such that for every complex neighborhood $W$ of $K$ with $\overline{W}\subseteq W_0$ there is $\delta>0$ with
	$$\{g\in\mathscr{H}^\infty(W);\,\|f-g\|_{\infty,W}<\delta\}\subseteq V,$$
	where $\|\cdot\|_{\infty,W}$ denotes the supremum norm over $W$. Because compact subsets of $\R^d$ are polynomially convex in $\C^d$ it follows from \cite[Theorem 2.7.7]{Hoermander2} that for any relatively compact, complex neighborhood $W$ of $K$ with $\overline{W}\subseteq W_0$ there is a (holomorphic) polynomial $p$ such that $\|f-p\|_{\infty,W}<\delta$. In particular, $p_{|X}\in V\cap\mathscr{A}(X)$. By the arbitrariness of $V$ and $f$ it follows that the restriction map $r_X^{X_n\cup\psi^m(X_n)}$ has dense range so that b3) of part i) in Theorem \ref{almost characterization of transitivity} is indeed fulfilled. Now, by the same theorem, iii) implies i) so that a) is proved.
	
	The proof of part b) is again done as the proof of a) by referring to Theorem \ref{almost characterization of mixing} instead of Theorem \ref{almost characterization of transitivity}.
\end{proof}

A generalization of the setting of real analytic functions is given in the next section.

\section{Spaces of ultradifferentiable functions}\label{ultradifferentiable}

In this section we consider as our function spaces defined by local properties spaces of ultradifferentiable functions on open subsets of $\R^d$, both of Roumieu type and Beurling type, and both quasianalytic classes as well as non-quasianalytic classes.

There are at least two ways to define spaces of ultradifferentiable functions. Classical Denjoy-Carleman classes are defined as smooth functions satisfying certain growth conditions on their Taylor coefficients while it was observed by Beurling \cite{Beurling} (see also Bj\"orck \cite{Bjoerck}) that one can also use decay properties with respect to a weight function of the Fourier transform of compactly supported smooth functions for this purpose as well. The latter approach was vastly generalized by Braun, Meise, and Taylor in \cite{BraunMeiseTaylor}. It is their approach to ultradifferentiable functions which we will follow here. For a comparison of these two approaches, see \cite{BonetMeiseMelikhov}; see also the article \cite{RainerSchindl1} by Rainer and Schindl.

Recall that a continuous increasing function $\omega:[0,\infty)\rightarrow [0,\infty)$ satisfying $\omega_{|[0,1]}=0$ is called a \textit{weight function} if the following properties hold:
\begin{enumerate}
	\item[$(\alpha)$] There is $K\geq 1$ such that $\omega(2t)\leq K(1+\omega(t))$ for all $t\geq 0$.
	\item[$(\beta)$] $\omega(t)=O(t)$ as $t$ tends to infinity.
	\item[$(\gamma)$] $\log(1+t)=o(\omega(t))$ as $t$ tends to infinity.
	\item[$(\delta)$] $\varphi:[0,\infty)\rightarrow[0,\infty), \varphi(x):=\omega(e^x)$ is convex.
\end{enumerate}
Recall that a weight function $\omega$ is called \textit{quasianalytic} if it satisfies the property
\begin{enumerate}
	\item[$(q)$] $\int_1^\infty\frac{\omega(t)}{t^2}dt=\infty$.
\end{enumerate}
A weight function which does not satisfy $(q)$ is called \textit{non-quasianalytic}. Because of ($\gamma$) and ($\delta$), for a weight function $\omega$ and $\varphi$ as in $(\delta)$, the \textit{Young conjugate} $\varphi^*$ of $\varphi$
$$\varphi^*:[0,\infty)\rightarrow[0,\infty),\varphi^*(y):=\sup_{x\geq 0}(yx-\varphi(x))$$
is well-defined, convex, increasing, and satisfies $\varphi^*(0)=0,\lim_{y\rightarrow\infty}\frac{y}{\varphi^*(y)}=0$, and $(\varphi^*)^*=\varphi$.
For a weight function $\omega$ and an open $X\subseteq\R^d$ we define
\begin{align*}
	\mathscr{E}_{(\omega)}(X)&:=\{f\in C^\infty(X):\forall\,K\subseteq X\mbox{ compact}\,\forall\,m\in\N:\\\
	&\|f\|_{(\omega),K,m}:=\sup_{x\in K}\sup_{\alpha\in\N_0^d}|\partial^\alpha f(x)|\exp\Big(-m\varphi^*\big(\frac{|\alpha|}{m}\big)\Big)<\infty\}
\end{align*}
and
\begin{align*}
	\mathscr{E}_{\{\omega\}}(X)&:=\{f\in C^\infty(X):\forall\,K\subseteq X\mbox{ compact}\,\exists\,m\in\N:\\\
	&\|f\|_{\{\omega\},K,m}:=\sup_{x\in K}\sup_{\alpha\in\N_0^d}|\partial^\alpha f(x)|\exp\Big(-\frac{1}{m}\varphi^*\big(m|\alpha|\big)\Big)<\infty\}.
\end{align*}
The elements of $\mathscr{E}_{(\omega)}(X)$, resp.\ $\mathscr{E}_{\{\omega\}}(X)$, are called \textit{$\omega$-ultradifferentiable functions of Beurling type on $X$}, resp.\ \textit{$\omega$-ultradifferentiable functions of Roumieu type on $X$}. Obviously, $\mathscr{E}_{(\omega)}(X)\subseteq\mathscr{E}_{\{\omega\}}(X)$ and clearly, $\mathscr{E}_{(\omega)}$ and $\mathscr{E}_{\{\omega\}}$ are sheaves on $\R^d$. $\mathscr{E}_{(\omega)}(X)$ contains non-trivial functions with compact support for some non-empty open $X$ if and only if $\omega$ is non-quasianalytic.\\

Prominent examples of weights are $\omega_\beta(t)=t^\beta$ with $0<\beta<1$ for which $\varphi^*(y)=y/\beta \log(y/e\beta)$ so that $\exp(-\varphi^*(|\alpha|\lambda)/\lambda)=(\lambda/\beta)^{-|\alpha|/\beta}(|\alpha|/e)^{-|\alpha|/\beta}$. By Stirling's formula $\mathscr{E}_{\{\omega_\beta\}}(X)$ is the classical Gevrey class of exponent $1/\beta$, $\Gamma^{1/\beta}(X)$, and $\mathscr{E}_{(\omega_\beta)}(X)$ is the so-called small Gevrey class of exponent $1/\beta$, $\gamma^{1/\beta}(X)$. The spaces $\Gamma^{1/\beta}(X)$ play an important role in the regularity theory of solutions of hypoelliptic partial differential equations, see \cite[Section 11.4]{Hoermander}.

Moreover, for the weight function $\omega(t)=t$ the corresponding Roumieu space $\mathscr{E}_{\{\omega\}}(X)$ coincides with $\mathscr{A}(X)$ while the corresponding Beurling space $\mathscr{E}_{(\omega)}(X)$ consist of the restrictions to $X$ of functions from $\mathscr{H}(\C^d)$.\\

As usual, $\mathscr{E}_{(\omega)}(X)$ will be equipped with the locally convex topology induced by the family  $\{\|\cdot\|_{(\omega),K,m};\,K\subseteq X\mbox{ compact},\,m\in\N\}$ of seminorms, and $\mathscr{E}_{\{\omega\}}(X)$ will be topologized as $\mbox{proj}_{\leftarrow Y}\mbox{ind}_{m\rightarrow\infty}\mathscr{E}_{\{\omega\}}(Y,m)$, where for each open, relatively compact subset $Y$ of $X$
\begin{align*}
	\mathscr{E}_{\{\omega\}}(Y,m):=&\{f\in C^\infty(Y);\\
	&\|f\|_{\{\omega\},Y,m}:=\sup_{x\in Y}\sup_{\alpha\in\N_0^d}|\partial^\alpha f(x)|\exp\Big(-\frac{1}{m}\varphi^*(m|\alpha|)\Big)<\infty\}
\end{align*}
endowed with the norm $\|\cdot\|_{\{\omega\},Y,m}$ is a Banach space.

\begin{proposition}\label{properties F1-F4 for ultradifferentiable}
	Let $\omega$ be a weight function. Equipped with their usual locally convex topologies the sheaves $\mathscr{E}_{(\omega)}$ and $\mathscr{E}_{\{\omega\}}$ on $\R^d$ both satisfy properties  $(\mathscr{F}1)-(\mathscr{F}4)$.
\end{proposition}

\begin{proof}
	It is well-known that for open $X\subseteq\R^d$ the space $\mathscr{E}_{(\omega)}(X)$ is a (nuclear) Fr\'echet space, \cite[Proposition 4.9]{BraunMeiseTaylor}. Obviously, point evaluations $\delta_x$ are continuous linear functionals on $\mathscr{E}_{(\omega)}(X)$ for any $x\in X$. Therefore, as observed in Remark \ref{properties of local function sheaves} i) the sheaf $\mathscr{E}_{(\omega)}$ on $\R^d$ satisfies $(\mathscr{F}1)$. Moreover, since $\mathscr{E}_{(\omega)}(X)$ is closed under differentiation and since polynomials obviously belong to $\mathscr{E}_{(\omega)}(X)$, properties $(\mathscr{F}2)-(\mathscr{F}4)$ are fulfilled, too.
	
	Clearly, in the definition of the topology of $\mathscr{E}_{\{\omega\}}(X)$ it is enough to take the projective limit with respect to an open, relatively compact exhaustion $(Y_n)_{n\in\N}$ of $X$. Therefore, being the projective limit of a sequence of $LB$-spaces it follows that $\mathscr{E}_{\{\omega\}}(X)$ is webbed. It has been shown recently by Debrouwere and Vindas \cite[Proposition 3.2]{DebrouwereVindas} that the space of ultra\-differentiable functions of Roumieu type is ultrabornological.
	
	Let $(X_n)_{n\in\N_0}$ be an open, relatively compact exhaustion of $X$. In order to show that the continuous bijection
	$$i:\mathscr{E}_{\{\omega\}}(X)\rightarrow\mbox{proj}_{\leftarrow n}(\mathscr{E}_{\{\omega\}}(X_n),r_{X_{n+1}}^{X_n}),f\mapsto (r_X^{X_n}(f))_{n\in\N_0}$$
	is open, let $V$ be an arbitrary zero neighborhood in $\mathscr{E}_{\{\omega\}}(X)$. By the definition of the topology on $\mathscr{E}_{\{\omega\}}(X)$ there is $n\in\N$ and zero neighborhood $U_n$ in $\mbox{ind}_{m\rightarrow\infty}\mathscr{E}_{\{\omega\}}(X_n,m)$ such that $\rho_n^{-1}(U_n)\subseteq V$ where
	$$\rho_n:\mathscr{E}_{\{\omega\}}(X)\rightarrow\mbox{ind}_{m\rightarrow\infty}\mathscr{E}_{\{\omega\}}(X_n,m), f\mapsto f_{|X_n}.$$
	With the continuous
	$$\tilde{\rho}_{n+1}:\mathscr{E}_{\{\omega\}}(X_{n+1})\rightarrow\mbox{ind}_{m\rightarrow\infty}\mathscr{E}_{\{\omega\}}(X_n,m), f\mapsto f_{|X_n}$$
	it follows that $\tilde{\rho}_{n+1}^{-1}(U_n)$ is a zero neighborhood in $\mathscr{E}_{\{\omega\}}(X_{n+1})$ for which
	$$\tilde{\rho}_{n+1}^{-1}(U_n)\cap r_X^{X_{n+1}}(\mathscr{E}_{\{\omega\}}(X))=r_X^{X_{n+1}}\Big(\rho_n^{-1}(U_n)\Big)\subseteq r_X^{X_{n+1}}\big(V\big).$$
	For $k\in\N_0$, let
	$$\pi_k:\mbox{proj}_{\leftarrow n}(\mathscr{E}_{\{\omega\}}(X_n),r_{X_{n+1}}^{X_n})\rightarrow\mathscr{E}_{\{\omega\}}(X_k), (f_n)_{n\in\N_0}\mapsto f_k$$
	so that
	\begin{align*}
		i(V)&=\pi_{n+1}^{-1}(r_X^{X_{n+1}}(V))\supseteq\pi_{n+1}^{-1}\big(\tilde{\rho}_{n+1}^{-1}(U_n)\cap r_X^{X_{n+1}}(\mathscr{E}_{\{\omega\}}(X))\big)\\
		&=\pi_{n+1}^{-1}\big(\tilde{\rho}_{n+1}^{-1}(U_n)\big).
	\end{align*}
	Since $\tilde{\rho}_{n+1}^{-1}(U_n)$ is a zero neighborhood in $\mathscr{E}_{\{\omega\}}(X_{n+1})$ the above inclusion implies that $i(V)$ is a zero neigborhood in $\mbox{proj}_{\leftarrow n}(\mathscr{E}_{\{\omega\}}(X_n),r_{X_{n+1}}^{X_n})$ so that $i$ is open and the sheaf $\mathscr{E}_{\{\omega\}}$ satisfies $(\mathscr{F}1)$.
	
	Properties $(\mathscr{F}2)-(\mathscr{F}4)$ of $\mathscr{E}_{\{\omega\}}$ follow again from the fact that $\mathscr{E}_{\{\omega\}}(X)$ is closed under differentiation and contains all polynomials.
\end{proof}

Since for arbitrary weight functions $\omega$ the spaces $\mathscr{E}_{(\omega)}(X)$ and $\mathscr{E}_{\{\omega\}}(X)$ are locally convex algebras (see \cite[Proposition 4.4]{BraunMeiseTaylor}) it follows that a weighted composition operator $C_{w,\psi}$ is well-defined on $\mathscr{E}_{(\omega)}(X)$ resp.\ $\mathscr{E}_{\{\omega\}}(X)$ whenever the weight $w$ belongs to the ultradifferentiable class and additionally, the composition with $\psi$ defines a continuous linear operator on $\mathscr{E}_{(\omega)}(X)$, resp.\ $\mathscr{E}_{\{\omega\}}(X)$. For non-quasianalytic weight functions $\omega$ this has been characterized by Fern\'andez and Galbis in \cite{FernandezGalbis} while Rainer and Schindl extended this characterization, among others, to more general weight functions in \cite{RainerSchindl1}, see also \cite{RainerSchindl2}. We define for a weight function $\omega$ property
\begin{enumerate}
	\item[$(\alpha_0)$] $\exists\,C>0, t_0>0\,\forall\lambda\geq 1, t\geq t_0:\,\omega(\lambda t)\leq C\lambda\omega(t)$.
\end{enumerate}
Property $(\alpha_0)$ characterizes when composition with a smooth function $\psi:X\rightarrow X$ with components $\psi_j, 1\leq j\leq d,$ all belonging to $\mathscr{E}_{\{\omega\}}(X)$ defines a continuous linear operator on $\mathscr{E}_{\{\omega\}}(X)$. If $\omega(t)=o(t)$ as $t$ tends to infinity, $(\alpha_0)$ characterizes  when composition with a smooth function $\psi:X\rightarrow X$ with components $\psi_j, 1\leq j\leq d,$ all belonging to $\mathscr{E}_{(\omega)}(X)$ is a continuous linear operator on $\mathscr{E}_{(\omega)}(X)$. Thus, under these conditions, $C_{w,\psi}$ is then a well-defined, continuous linear operator which then also acts locally on $\mathscr{E}_{(\omega)}(X)$ resp.\ $\mathscr{E}_{\{\omega\}}(X)$.

\begin{theorem}\label{dynamics in Beurling}
	Let $\omega$ be a weight function satisfying $(\alpha_0)$ and such that $\omega(t)=o(t)$ as $t$ tends to infinity. Moreover, let $X\subseteq\R^d$ be open, $w\in\mathscr{E}_{(\omega)}(X)$, and $\psi:X\rightarrow X$ be smooth such that $\psi_j\in\mathscr{E}_{(\omega)}(X)$ for all $1\leq j\leq d$.
	\begin{itemize}
	\item[a)] The following are equivalent.
	\begin{itemize}
		\item[i)] $C_{w,\psi}$ is weakly mixing on $\mathscr{E}_{(\omega)}(X)$.
		\item[ii)] $C_{w,\psi}$ is hypercyclic on $\mathscr{E}_{(\omega)}(X)$.
		\item[iii)] $w$ has no zeros, $\psi$ is injective, run-away, and $\mbox{det}J\psi(x)\neq 0$ for all $x\in X$.
	\end{itemize}
	\item[b)] The following are equivalent.
	\begin{itemize}
		\item[i)] $C_{w,\psi}$ is mixing on $\mathscr{E}_{(\omega)}(X)$.
		\item[ii)] $w$ has no zeros, $\psi$ is injective, strong run-away, and $\mbox{det}J\psi(x)\neq 0$ for all $x\in X$.
	\end{itemize}
	\end{itemize}
\end{theorem}

\begin{proof}
	By a result due to Heinrich and Meise \cite[Proposition 3.2]{HeinrichMeise} $\mathscr{H}(\C^d)$ is dense in $\mathscr{E}_{(\omega)}(X)$. In particular, (holomorphic) polynomials are dense in $\mathscr{E}_{(\omega)}(X)$ implying that the latter Fr\'echet space is separable. Because polynomials are contained in $\mathscr{E}_{(\omega)}(\R^d)$ it follows that the latter space is dense in $C(\R^d)$. By Theorem \ref{almost characterization of transitivity} it thus follows that i) implies ii) and that ii) implies iii) in part a).
	
	If iii) in a) is satisfied, it follows from the hypotheses on $\omega$ and \cite[Theorem 4]{RainerSchindl2} that $\frac{1}{w}\in\mathscr{E}_{(\omega)}(X)$ and that for any $m\in\N$ the components of the smooth function $(\psi^m)^{-1}:X\rightarrow X$ belong to $\mathscr{E}_{(\omega)}(X)$. Therefore, applying again \cite[Theorem 4]{RainerSchindl2} it follows that for every open subset $Y$ of $\psi^m(X)$ and any $f\in\mathscr{E}_{(\omega)}((\psi^m)^{-1}(Y))$ the function
	$$\tilde{f}:Y\rightarrow\K,y\mapsto\big(\frac{f}{\prod_{j=0}^{m-1} w(\psi^j(\cdot))}\big)\circ(\psi^m)^{-1}(y)$$
	belongs to $\mathscr{E}_{(\omega)}(Y)$. As detailed in Remark \ref{considering the almost characterization} iii) this implies that condition a) of part i) of Theorem \ref{almost characterization of transitivity} is satisfied. Moreover, because $\psi$ is run-away and $\mbox{det}J\psi(x)\neq 0$ for every $x\in X$ it follows that conditions b1) and b2) from part i) of Theorem \ref{almost characterization of transitivity} are fulfilled for any open, relatively compact exhaustion $(X_n)_{n\in\N}$ of $X$. Finally, applying \cite[Proposition 3.2]{HeinrichMeise} once more it follows in particular that condition b3) from part i) of Theorem \ref{almost characterization of transitivity} is satisfied, too, for an arbitrary open, relatively compact, exhaustion $(X_n)_{n\in\N}$ of $X$. Hence,  by Theorem \ref{almost characterization of transitivity}, iii) implies i) in part a).
	
	Mutatis mutandis, part b) of the theorem is again proved by applying Theorem \ref{almost characterization of mixing} instead of Theorem \ref{almost characterization of transitivity}.
\end{proof}

\begin{theorem}\label{dynamics in Roumieu}
	Let $\omega$ be a weight function satisfying $(\alpha_0)$. Moreover, let $X\subseteq\R^d$ be open, $w\in\mathscr{E}_{\{\omega\}}(X)$, and $\psi:X\rightarrow X$ be smooth such that $\psi_j\in\mathscr{E}_{\{\omega\}}(X)$ for all $1\leq j\leq d$.
	\begin{itemize}
		\item[a)] The following are equivalent.
		\begin{itemize}
			\item[i)] $C_{w,\psi}$ is weakly mixing on $\mathscr{E}_{\{\omega\}}(X)$.
			\item[ii)] $C_{w,\psi}$ is transitive on $\mathscr{E}_{\{\omega\}}(X)$.
			\item[iii)] $w$ has no zeros, $\psi$ is injective, run-away, and $\mbox{det}J\psi(x)\neq 0$ for all $x\in X$.
		\end{itemize}
		\item[b)] The following are equivalent.
		\begin{itemize}
			\item[i)] $C_{w,\psi}$ is mixing on $\mathscr{E}_{\{\omega\}}(X)$.
			\item[ii)] $w$ has no zeros, $\psi$ is injective, strong run-away, and $\mbox{det}J\psi(x)\neq 0$ for all $x\in X$.
		\end{itemize}
	\end{itemize}
\end{theorem}

\begin{proof}
	We first prove part a) of the theorem. Clearly, i) implies ii) and since polynomials are contained in $\mathscr{E}_{\{\omega\}}(\R^d)$ the latter space is dense in $C(\R^d)$ so that by Theorem \ref{almost characterization of transitivity}, iii) follows from ii). If iii) is satisfied, it follows as in the proof of Theorem \ref{dynamics in Beurling} from \cite[Theorem 3]{RainerSchindl2}, and Remark \ref{considering the almost characterization} iii) that condition a) in part i) of Theorem \ref{almost characterization of transitivity} is fulfilled. Condition b3) of part i) in Theorem \ref{almost characterization of transitivity} is satisfied for any open, relatively compact exhaustion $(X_n)_{n\in\N}$ because by \cite[Proposition 3.2]{HeinrichMeise} $\mathscr{H}(\C^d)$ is dense in $\mathscr{E}_{\{\omega\}}(Y)$ for every open subset $Y$ of $\R^d$. From the run-away property and the injectivity of $\psi$ together with $\mbox{det}J(x)\neq 0$ for all $x\in X$ it follows that conditions b1) and b2) from part i) of Theorem \ref{almost characterization of transitivity} are satisfied, too, so that i) follows.
	
	The proof of part b) is once more a straight forward modification of the proof of part a) involving Theorem \ref{almost characterization of mixing}.
\end{proof}

Since for the weight function $\omega(t)=t$ the spaces $\mathscr{E}_{\{\omega\}}(X)$ and $\mathscr{A}(X)$ coincide as locally convex spaces it follows that Corollary \ref{real analytic case} is a special case of Theorem \ref{dynamics in Roumieu}.

\section{Kernels of elliptic differential operators}\label{kernels of elliptic differential operators}

In this section we apply the results from section \ref{dynamical properties} to weighted composition operators defined on kernels of elliptic partial differential operators. The special case of the Cauchy-Riemann operator will give the space of holomorphic functions of a single variable equipped with the compact-open topology. In this context dynamical properties of (even a sequence of) unweighted composition operators have been studied by Bernal-Gonz\'alez, Montes-Rodr\'iguez \cite{BernalMontes}, resp.\ Gro{\ss}e-Erdmann, Mortini \cite{Grosse-ErdmannMortini}. For dynamical properties of weighted composition operators on the Fr\'echet space of holomorphic functions see also the articles \cite{YousefiRezaei} and \cite{Bes}.

The special case of the Laplace operator gives the space of harmonic functions endowed with the compact-open topology where dynamical properties of special unweighted composition operators have been studied for example in \cite{Dzagnidze}, \cite{Armitage}. The results in this section complement those from \cite{CalderonMueller}, \cite{KalmesNiess}, and \cite{KalmesNiessRansford} where hypercyclicity of special unweighted composition operators on spaces of zero solutions to linear partial differential equations with constant coefficients is considered.

As explained in example \ref{examples of sheaves} v), for a non-constant polynomial with complex coefficients in $d\geq 2$ variables $P\in\C[X_1,\ldots,X_d]$ and an open subset $X\subseteq\R^d$ we define
$$C_P^{\infty}(X):=\{u\in C^\infty(X);\,P(\partial)u=0\mbox{ in }X\},$$
where for $P(\xi)=\sum_{|\alpha|\leq m}a_\alpha\xi^\alpha$ with $a_{\alpha_0}\neq 0$ for some multiindex $\alpha_0\in\N_0^d$ with $|\alpha_0|(=\alpha_1+\ldots+\alpha_d)=m$ we define
$$\forall\,u\in C^\infty(X), x\in X:\,P(\partial)u(x)=\sum_{|\alpha|\leq m}a_\alpha\partial^\alpha u(x).$$

As a closed subspace of the separable nuclear Fr\'echet space $C^\infty(X)$ the space $C_P^\infty(X)$ is then again a separable nuclear Fr\'echet space. For hypoelliptic polynomials $P$ - by definition - for every open $X\subseteq\R^d$ the spaces $C_P^\infty(X)$ and
$$\mathscr{D}'_P(X):=\{u\in\mathscr{D}'(X);\,P(\partial)u=0\mbox{ in }X\}$$
coincide (that is, every distribution $u$ on $X$ which satisfies $P(\partial)u=0$ in $X$ is already a smooth function). By a result of Malgrange (see e.g.\ \cite[Theorem 52.1]{Treves}) the spaces $C_P^\infty(X)$ and $\mathscr{D}'_P(X)$ also coincide as locally convex spaces when the latter is endowed with the relative topology inherited from $\mathscr{D}'(X)$ equipped with the strong dual topology as the topological dual of $\mathscr{D}(X)$. This implies in particular, that for hypoelliptic polynomials the compact-open topology on $C^\infty_P(X)$ and the relative topology inherited from $C^\infty(X)$ coincide. Therefore, for hypoelliptic polynomials $P$ the space $C_P^\infty(X)$ endowed with the compact-open topology is a separable (nuclear) Fr\'echet space for every open $X\subseteq\R^d$.

As already mentioned in example \ref{examples of sheaves} v), $C_P^\infty$ defines a sheaf on $\R^d$ which satisfies $(\mathscr{F}1)$ and $(\mathscr{F}2)$ but generally $(\mathscr{F}3)$ need not hold. However, the next proposition shows that for hypoelliptic polynomials $P$ both $(\mathscr{F}3)$ and $(\mathscr{F}4)$ hold for $C_P^\infty$.

\begin{proposition}\label{f3 and f4 for kernel}
	Let $d\geq 2$ and let $P\in\C[X_1,\ldots,X_d]$ be hypoelliptic. The sheaf $C_P^\infty$ satisfies both $(\mathscr{F}3)$ and $(\mathscr{F}4)$.
\end{proposition} 

\begin{proof}
	Fix $x,y\in\R^d$ with $x\neq y$. By renumbering the coordinates we can assume without loss of generality that $x_d-y_d\neq 0$. For $\xi'\in\R^{d-1}$ we denote by $\lambda_1(\xi'),\ldots,\lambda_{l(\xi')}(\xi')\in\C$ the pairwise distinct roots of the polynomial
	$$\C\rightarrow\C,z\mapsto P(\xi',z),$$
	ordered in such a way that $(\mbox{Im}\lambda_j(\xi'))_{1\leq j\leq l(\xi')}$ is increasing and $\mbox{Re}\lambda_j(\xi')<\mbox{Re}\lambda_{j+1}(\xi')$ whenever $\mbox{Im}\lambda_j(\xi')=\mbox{Im}\lambda_{j+1}(\xi')$.
	
	Then the mapping
	$$\mbox{Im}\lambda_1:\R^{d-1}\rightarrow\R,\xi'\mapsto\mbox{Im}\lambda_1(\xi')$$
	is continuous. Indeed, fix $\xi'_0\in\R^{d-1}$ and let $m_j$ be  the multiplicities of the $\lambda_j(\xi_0')$. Let $\varepsilon>0$. Without loss of generality we assume that $\varepsilon$ is so small that for every $j$ we have $\mbox{Im}\lambda_j(\xi_0')+\varepsilon< \mbox{Im}\lambda_{j+1}(\xi_0')-\varepsilon$ if $\mbox{Im}\lambda_j(\xi_0')<\mbox{Im}\lambda_{j+1}(\xi_0')$ and $\mbox{Re}\lambda_j(\xi_0')+\varepsilon< \mbox{Re}\lambda_{j+1}(\xi_0')-\varepsilon$ if $\mbox{Im}\lambda_j(\xi_0')=\mbox{Im}\lambda_{j+1}(\xi_0')$. Then $B(\lambda_j(\xi_0'),\varepsilon)\cap B(\lambda_k(\xi_0'),\varepsilon)=\emptyset$ for every $1\leq j,k\leq l(\xi_0'), j\neq k$. For any $1\leq j\leq l(\xi_0')$ we have by Taylor's Theorem for every $\xi'\in\R^{d-1}$ with $|\xi'-\xi_0'|<1$
	\begin{align*}
		\forall\, z\in\C, |z-\lambda_j(\xi'_0)|&=\varepsilon:|P(\xi',z)-P(\xi_0',z)|\\
		&=|\sum_{\alpha\neq 0}P^{(\alpha)}((\xi_0',0)+ze_d) \frac{(\xi'-\xi_0',0)^\alpha}{\alpha!}|\\
		&\leq\sum_{\alpha\neq 0}|P^{(\alpha)}((\xi_0',0)+ze_d)| \frac{|\xi'-\xi_0'|^{|\alpha|}}{\alpha!}\\
		&\leq |\xi'-\xi_0'|\sup_{\zeta\in\C, |\zeta-\lambda_j(\xi_0')|=\varepsilon} \sum_{\alpha\neq 0}| P^{(\alpha)}((\xi_0',0)+\zeta e_d)|\frac{1}{\alpha!}.
	\end{align*}
	Thus, if $|\xi'-\xi_0'|$ is sufficiently small the right hand side of the above inequality is less than
	\[\inf\{|P(\xi_0',\zeta)|;\,\zeta\in\C, |\zeta-\lambda_j(\xi_0')|=\varepsilon\}\big(\leq |P(\xi_0',z)|\;(|z-\lambda_j(\xi_0')|=\varepsilon)\big)\]
	for any $j$. Hence it follows from Rouch\'e's Theorem for $\xi'$ sufficiently close to $\xi_0'$, say $|\xi'-\xi_0'|<\delta$, that $P(\xi',\cdot)$ has exactly $m_j$ roots $z_{j,1}^{\xi'},\ldots z_{j,m_j}^{\xi'}$ in $B(\lambda_j(\xi_0'),\varepsilon)$ for each $1\leq j\leq l(\xi_0')$.
	
	Now set $k:=\max\{1\leq j\leq l(\xi_0');\,\mbox{Im}\lambda_1(\xi_0')=\mbox{Im}\lambda_j(\xi_0')\}$. Then for any $1\leq j\leq k, 1\leq r\leq m_j$ we have
	\[|\mbox{Im} z_{j,r}^{\xi'}-\mbox{Im}\lambda_j(\xi_0')|<\varepsilon\]
	if $|\xi'-\xi_0'|<\delta$ and according to our choice of $\varepsilon$ we have
	\begin{align*}
	&\max\{\mbox{Im} z_{j,r}^{\xi'};\,1\leq j\leq k,1\leq r\leq m_j\}\\
	&<\min\{\mbox{Im} z_{j,r}^{\xi'};\,k+1\leq j\leq k,1\leq r\leq m_j\}.
	\end{align*}
	Therefore $\mbox{Im}\lambda_1(\xi'),\ldots, \mbox{Im}\lambda_{m_1}(\xi'),\mbox{Im}\lambda_{m_1+1}(\xi'),\ldots,\mbox{Im}\lambda_{m_k}(\xi')$ all belong to
	$$(\mbox{Im}\lambda_1(\xi_0')-\varepsilon,\mbox{Im}\lambda_1(\xi_0')+\varepsilon)=\ldots=(\mbox{Im} \lambda_k(\xi_0')-\varepsilon,\mbox{Im} \lambda_k(\xi_0')+\varepsilon)$$
	so that for all $1\leq j\leq m_k$ we have $|\mbox{Im}\lambda_j(\xi')-\mbox{Im}\lambda_1(\xi_0')|<\varepsilon$ when $|\xi'-\xi_0'|<\delta$. In particular, $|\lambda_1(\xi')-\lambda_1(\xi_0')|<\varepsilon$ whenever $|\xi'-\xi_0'|<\delta$ which gives the continuity of $\mbox{Im}\lambda_1$.
	
	Denoting $V(P):=\{\zeta\in\C^d;\,P(\zeta)=0\}$ it follows that $(\xi',\lambda_1(\xi'))\in V(P)$ for each $\xi'\in\R^{d-1}$ and therefore
	\begin{equation}\label{distance inequality}
		\forall\,\xi'\in\R^{d-1}:\,\mbox{dist}\big((\xi',\mbox{Re}\lambda_1(\xi')),V(P)\big)\leq |\mbox{Im}\lambda_1(\xi')|.
	\end{equation}
	Since $P$ is hypoelliptic we have
	$$\lim_{x\in\R^d,|x|\rightarrow\infty}\mbox{dist}(x,V(P))=\infty$$
	(see \cite[Theorem 11.1.3]{Hoermander}) which combined with (\ref{distance inequality}) yields
	\begin{equation}\label{imaginary parts of zeros}
	\lim_{|\xi'|\rightarrow\infty}|\mbox{Im}\lambda_1(\xi')|=\infty.
	\end{equation}
	This implies in particular that there are $\xi',\eta'\in\R^{d-1}$ for which
	\begin{align*}
		&\langle\mbox{Im}\big((\xi',\lambda_1(\xi'))-(\eta',\lambda_1(\eta'))\big), x-y\rangle=\mbox{Im}(\lambda_1(\xi')-\lambda_1(\eta'))(x_d-y_d)\\
		&\notin\{2\pi k;\,k\in\Z\}.
	\end{align*}
	Thus, there are $\zeta_1,\zeta_2\in V(P)$ such that $\langle\mbox{Im}\,(\zeta_1-\zeta_2),x-y\rangle$ is not an integer multiple of $2\pi$. Setting $\langle\eta,v\rangle=\sum_{j=1}^d\eta_j v_j$ for $\eta, v\in\C^d$, the function
	$$g:\R^d\rightarrow\C, g(w):=\exp(\langle\zeta_1,x\rangle+\langle\zeta_2,w\rangle)-\exp(\langle\zeta_2,x\rangle+\langle\zeta_1,w\rangle)$$
	satisfies $g\in C_P^\infty(\R^d)$, $g(x)=0$, and
	$$g(y)=\exp(\langle\zeta_1,y\rangle+\langle\zeta_2,x\rangle)\Big(\exp\big(\langle\zeta_1-\zeta_2,x-y\rangle\big)-1\Big)\neq 0$$
	which implies the existence of $f\in C_P^\infty(\R^d)$ with $f(x)=0$ and $f(y)=1$. Hence, $(\mathscr{F}3)$ is satisfied.
	
	To verify $(\mathscr{F}4)$, let $X\subseteq\R^d$ be open. We first observe that $-\partial_j\delta_x, 1\leq j\leq d,$ is a continuous linear functional on $C_P^\infty(X)$. Now, let $h\in\R^d\backslash\{0\}$ and $\lambda\in\K$. By renumbering the coordinates if necessary we may assume that $h_d\neq 0$. By (\ref{imaginary parts of zeros}) it follows
	\begin{equation}\label{not zero}
		\exists\,\zeta\in V(P):\,\langle h,\mbox{Im}\zeta\rangle -\mbox{Im}\lambda\neq 0.
	\end{equation}
	Because $e_\zeta\in C_P^\infty(X)$, where $e_\zeta(x):=\exp(\langle\zeta,x\rangle)$,
	$$\sum_{j=1}^d h_j\partial_j e_\zeta(x)-\lambda e_\zeta(x)=\Big(\langle h,\mbox{Re}\zeta\rangle-\mbox{Re}\lambda+i(\langle h,\mbox{Im}\zeta\rangle-\mbox{Im}\lambda)\Big)e_\zeta(x)$$
	where the factor
	$$\Big(\langle h,\mbox{Re}\zeta\rangle-\mbox{Re}\lambda+i(\langle h,\mbox{Im}\zeta\rangle-\mbox{Im}\lambda)\Big)$$
	does not vanish by (\ref{not zero}). Therefore the continuous linear functional
	$$u\mapsto \sum_{j=1}^d h_j\partial_j u-\lambda u$$
	on $C_P^\infty(X)$ does not vanish identically so that $(\mathscr{F}4)$ is fulfilled.
\end{proof}

For us, elliptic polynomials will be of particular interest. Recall that a polynomial $P\in\C[X_1,\ldots,X_d]$, $P(\xi)=\sum_{|\alpha|\leq m}a_\alpha\xi^\alpha$ is called elliptic if
$$\forall\,\xi\in\R^d\backslash\{0\};\,P_m(\xi)\neq 0,$$
where $P_m(\xi)=\sum_{|\alpha|=m}a_\alpha \xi^\alpha$ denotes the principal part of $P$. As is well-known, elliptic polynomials are hypoelliptic (see e.g.\ \cite[Theorem 11.1.10]{Hoermander}). In particular, identifying $\C$ as usual with $\R^2$, and choosing for $P\in\C[X_1,X_2]$ the polynomial $P(\xi_1,\xi_2)=\frac{1}{2}(\xi_1+i \xi_2)$ gives the Cauchy-Riemann operator $\partial_{\bar{z}}$ and we have that $C_P^\infty(X)=\mathscr{H}(X)$ holds as locally convex spaces for any open $X\subseteq\C$ so that the sheaf of holomorphic functions (equipped with the compact-open topology) on open subsets of $\C$ is a special case.

Arguably the most prominent elliptic differential operator, apart from the Cauchy-Riemann operator is the Laplace operator. Thus, the sheaf of harmonic functions (equipped with the compact-open topology) on open subsets of $\R^d$ is also a special case of the sheaves $C_P^\infty$.\\

We are now going to characterize when for an elliptic polynomial $P$ and an open $X\subseteq\R^d$ a well-defined weighted composition operator $C_{w,\psi}$ on $C_P^\infty(X)$ is weakly mixing. As follows in particular from the results obtained in \cite{Grosse-ErdmannMortini} an unweighted composition operator cannot be hypercyclic on $\mathscr{H}(X)$ if $X$ is a finitely connected but not simply connected domain. Thus, the special case of the Cauchy-Riemann operator shows that topological properties of $X$ have to be taken into account.  

\begin{theorem}\label{dynamics on elliptic kernels}
	Let $P$ be an elliptic polynomial and let $X\subseteq\R^d$ be open and homeo\-morphic to $\R^d$. Moreover, let $w:X\rightarrow\C$ and $\psi:X\rightarrow X$ be smooth such that $C_{w,\psi}$ is well-defined on $C_P^\infty(X)$ and acts locally on $C_P^\infty(X)$.
	\begin{itemize}
		\item[a)] The following are equivalent.
		\begin{itemize}
			\item[i)] $C_{w,\psi}$ is weakly mixing on $C_P^\infty(X)$.
			\item[ii)] $C_{w,\psi}$ has dense range, $w$ has no zeros, and $\psi$ is injective and run-away.
			\item[iii)] $w$ has no zeros, $\psi$ is injective and run-away, and for each $m\in\N_0$ and every open, relatively compact $Y\subseteq\psi^m(X)$ it holds
			$$r_X^{(\psi^m)^{-1}(Y)}(C_P^\infty(X))\subseteq\overline{\big(C_{w,\psi,(\psi^{m-1})^{-1}(Y)}\circ\ldots\circ C_{w,\psi,Y}\big)(C_P^\infty(Y))},$$
			where the closure is taken in $C_P^\infty((\psi^m)^{-1}(Y))$.
		\end{itemize}
		Moreover, $\mbox{det}J\psi(x)\neq 0$ for all $x\in X$ can be added to ii) and iii). If additionally $|w(x)|\leq 1$ for all $x\in X$ then the above are equivalent to
		\begin{itemize}
			\item[iv)] $C_{w,\psi}$ is hypercyclic on $C_P^\infty(X)$.
		\end{itemize}
		\item[b)] The following are equivalent.
		\begin{itemize}
			\item[i)] $C_{w,\psi}$ is mixing on $C_P^\infty(X)$.
			\item[ii)] $C_{w,\psi}$ has dense range, $w$ has no zeros, and $\psi$ is injective and strong run-away.
			\item[iii)] $w$ has no zeros, $\psi$ is injective and strong run-away, and for each $m\in\N_0$ and every open, relatively compact $Y\subseteq\psi^m(X)$ it holds
			$$r_X^{(\psi^m)^{-1}(Y)}(C_P^\infty(X))\subseteq\overline{\big(C_{w,\psi,(\psi^{m-1})^{-1}(Y)}\circ\ldots\circ C_{w,\psi,Y}\big)(C_P^\infty(Y))},$$
			where the closure is taken in $C_P^\infty((\psi^m)^{-1}(Y))$.
		\end{itemize}
		Moreover, $\mbox{det}J\psi(x)\neq 0$ for all $x\in X$ can be added to ii) and iii).
	\end{itemize}
\end{theorem}

For the proof of Theorem \ref{dynamics on elliptic kernels} some preparations have to be made. Before providing these let us mention that in case of $\C=\R^2$ it follows from the Riemann Mapping Theorem that every simply connected, connected, open $X\subseteq\C$ different from $\C$ is in particular homeomorphic to the open unit disc in $\C$ which itself is homeomorphic to $\C$. Thus, in case of $d=2$ the topological hypothesis on $X$ in Theorem \ref{dynamics on elliptic kernels} means precisely that $X$ is a simply connected domain.

In order to prove Theorem \ref{dynamics on elliptic kernels} we need the following version of the celebrated Jordan-Brouwer Separation Theorem that can be found in \cite[Satz 5.23]{Mayer}. Since this reference is written in German and since we could not find a different reference we include a proof here - which is different from the one presented in \cite{Mayer} - for the reader's convenience.

\begin{theorem}\label{Version of the Jordan-Brouwer}{\rm\textbf{(Version of Jordan-Brouwer Separation Theorem)}} Let $K_1$ and $K_2$ be homeomorphic compact subsets of $\R^d$. Then $\R^d\backslash K_1$ and $\R^d\backslash K_2$ have the same number of connected components.
\end{theorem}

\begin{proof}
	For a topological space $X$ we denote as usual the $n$-th reduced homology group, respectively cohomology group, with coefficients in $\Z$ by $\tilde{H}_n(X,\Z)$ respectively $\tilde{H}^n(X,\Z)$. Moreover, let $S^d$ be the unit sphere in $\R^{d+1}$, $N:=(1,0,\ldots,0)\in S^d$ be the "north pole" and let $\phi:\R^d\rightarrow S^d\backslash\{N\}$ be a homeomorphism. Since $K_1$ and $K_2$ are homeomorphic it follows that the compact subsets $\{N\}\dot{\cup}\phi(K_1)$ and $\{N\}\dot{\cup}\phi(K_2)$ of $S^d$ are homeomorphic. Thus, the groups $\tilde{H}^{d-1}(\{N\}\dot{\cup}\phi(K_1),\Z)$ and $\tilde{H}^{d-1}(\{N\}\dot{\cup}\phi(K_2),\Z)$ are isomorphic. Using Alexander Duality (see \cite[Theorem 3.44]{Hatcher}) it follows that the groups $\tilde{H}_0(S^d\backslash(\{N\}\dot{\cup}\phi(K_1)),\Z)$ and $\tilde{H}_0(S^d\backslash(\{N\}\dot{\cup}\phi(K_2)),\Z)$ are isomorphic and therefore, the same is true for the groups $\tilde{H}_0(\R^d\backslash K_1,\Z)$ and $\tilde{H}_0(\R^d\backslash K_2,\Z)$. Hence, the groups $\tilde{H}_0(\R^d\backslash K_1,\Z)\oplus\Z$ and $\tilde{H}_0(\R^d\backslash K_2,\Z)\oplus\Z$ are isomorphic as well. Since for any topological space $X$ the groups $\tilde{H}_0(X,\Z)\oplus\Z$ and $H_0(X,\Z)$, the homology group of degree zero of $X$ are isomorphic (see e.g.\ \cite[page 110]{Hatcher}) and since $H_0(X,\Z)$ is isomorphic to $\oplus_{\alpha\in C(X)}\Z$, where $C(X)$ is the set of all pathwise connected components of $X$, it follows that $\R^d\backslash K_1$ and $\R^d\backslash K_2$ have the same number of pathwise connected components. Since $\R^d\backslash K_j$ are open in $\R^d$ and thus locally pathwise connected it follows that $\R^d\backslash K_1$ and $\R^d\backslash K_2$ have indeed the same number of connected components.
\end{proof}

\begin{proposition}\label{obvious topological facts}
	Denoting for $x\in\R^d$ and $\varepsilon>0$ the open, resp.\ closed ball about $x$ with radius $\varepsilon$ by $B(x,\varepsilon)$ and $B[x,\varepsilon]$, respectively, for every continuous and injective $\psi:\R^d\rightarrow\R^d$ the following hold.
	\begin{itemize}
		\item[i)] $\forall\,n\in\N:\,\overline{\R^d\backslash \psi(B[0,n])}=\R^d\backslash\psi(B(0,n)),\,\overline{\R^d\backslash B[0,n]}=\R^d\backslash B(0,n).$
		\item[ii)] If $d\geq 2$ then $\R^d\backslash\psi(B[0,n])$ and $\R^d\backslash\psi(B(0,n))$ are connected for every $n\in\N$.
		\item[iii)] If $d\geq 2$ and $n\in\N$ then $\R^d\backslash\big(\psi(B(0,n))\cup B(0,n)\big)$ is connected whenever $\psi(B[0,n])\cap B[0,n]=\emptyset$.
	\end{itemize}
\end{proposition}

\begin{proof}
	Denoting the interior of $A\subseteq\R^d$ by $\mbox{int}(A)$ we have $\R^d\backslash\mbox{int}(A)=\overline{\R^d\backslash A}$ for any $A\subseteq\R^d$. 
	Thus, $\overline{\R^d\backslash\psi(B[0,n])}=\R^d\backslash\mbox{int}(\psi(B[0,n]))$. Since $\psi$ is continuous and injective it follows from Brouwer's Invariance of Domain Theorem that $\psi(B(0,n))$ is open in $\R^d$. Thus, $\psi(B(0,n))\subseteq\mbox{int}(\psi(B[0,n]))$. On the other hand, for $x$ in the interior of $\psi(B[0,n])$ there is $\delta>0$ 
	$$\psi^{-1}(B(x,\delta))\subseteq\psi^{-1}\big(\psi(B[0,n])\big)=B[0,n],$$
	where we used the injectivity of $\psi$. Since $\psi^{-1}(B(x,\delta))$ is open in $\R^d$, we conclude $\psi^{-1}(B(x,\delta))\subseteq B(0,n)$. From
	$$B(x,\delta)\subseteq\psi(B[0,n])\subseteq\psi(\R^d)$$
	and the injectivity of $\psi$ we get
	$$B(x,\delta)=\psi\big(\psi^{-1}(B(x,\delta))\big)\subseteq\psi(B(0,n)).$$
	Since $x\in\mbox{int}(\psi(B[0,n]))$ was chosen arbitrarily it follows $\mbox{int}(\psi(B[0,n]))\subseteq\psi(B(0,n))$ so that $\mbox{int}(\psi(B[0,n]))=\psi(B(0,n))$ which proves i).
	
	In order to prove ii), we define for $n\in\N$
	$$\psi_n:B[0,n]\rightarrow \psi(B[0,n]),x\mapsto \psi(x)$$
	which is a continuous bijection, thus a homeomorphism due to the compactness of $B[0,n]$. $\R^d\backslash B[0,n]$ is connected because $d\geq 2$ so that by Theorem \ref{Version of the Jordan-Brouwer} the same is true for $\R^d\backslash\psi(B[0,n])$. Therefore, using i), it follows that
	$$\overline{\R^d\backslash\psi(B[0,n])}=\R^d\backslash \psi(B(0,n))$$
	is connected, too, which proves ii).
	
	In order to show iii), we first observe that
	$$\overline{\R^d\backslash\big(\psi(B[0,n])\cup B(0,n)\big)}=\R^d\backslash\mbox{int}\big(\psi(B[0,n])\cup B(0,n)\big).$$
	Clearly,
	$$\mbox{int}(\psi(B[0,n])\cup B(0,n))\supseteq\mbox{int}(\psi(B[0,n]))\cup B(0,n)$$
	and because $\psi(B[0,n])$ and $B[0,n]$ are disjoint closed sets we also have
	$$\mbox{int}(\psi(B[0,n])\cup B(0,n))\subseteq\mbox{int}\big(\psi(B[0,n])\big)\cup B(0,n)$$
	which combined with $\mbox{int}(\psi(B[0,n]))=\psi(B(0,n))$ gives
	$$\overline{\R^d\backslash\big(\psi(B[0,n])\cup B(0,n)\big)}=\R^d\backslash\big(\psi(B(0,n))\cup B(0,n)\big).$$
	Because the closure of a connected set is connected, it suffices to show the connectedness of the set $\R^d\backslash\big(\psi(B[0,n])\cup B(0,n)\big).$ Let $x,y$ be in the complement of $\psi(B[0,n])\cup B(0,n)$. By ii), $\R^d\backslash\psi(B[0,n])$ is connected. Because open, connected subsets of $\R^d$ are pathwise connected, there is a continuous $\gamma:[0,1]\rightarrow\R^d\backslash\psi(B[0,n])$ with $\gamma(0)=x$ and $\gamma(1)=y$. Without loss of generality, we can assume that $\gamma([0,1])$ does not intersect $B(0,n)$ because otherwise let
	$$t_0:=\inf\{t\in[0,1];\,\gamma(t)\in B(0,n)\},\; t_1:=\sup\{t\in[0,1];\,\gamma(t)\in B(0,n)\}.$$
	Then $0<t_0\leq t_1<1$ and $\gamma(t_j)\in\partial B(0,n), j=0,1$. Since $d\geq 2$ the set $\partial B(0,n)$ is pathwise connected so there is a continuous $\alpha:[t_0,t_1]\rightarrow\partial B(0,n)$ such that $\alpha(t_j)=\gamma(t_j), j=0,1$. Then
	$$\tilde{\gamma}:[0,1]\rightarrow\R^d\backslash(\psi(B[0,n]))\cup B(0,n)),t\mapsto\begin{cases}
	\gamma(t), &t\notin[t_0,t_1]\\ \alpha(t), &t\in[t_0,t_1]
	\end{cases} $$
	is a well-defined continuous mapping with $\tilde{\gamma}(0)=x$ and $\tilde{\gamma}(1)=y$.
	
	This shows that $\R^d\backslash\big(\psi(B[0,n])\cup B(0,n)\big)$ is pathwise connected, a fortiori connected which proves iii). 
\end{proof}

\begin{proposition}\label{preparation Lax-Malgrange}
	Let $d\geq 2$ and $X\subseteq\R^d$ be open and homeomorphic to $\R^d$ as well as $\psi:X\rightarrow X$ be continuous, injective, and run-away. Then there is a relatively compact-open exhaustion $(X_n)_{n\in\N}$ of $X$ such that
	$$\forall\,n\in\N\,\exists\,m\in\N:\,X_n\cap\psi^m(X_n)=\emptyset$$
	and if $m,n\in\N$ are such that $X_n$ and $\psi^m(X_n)$ are disjoint, if $X\backslash(X_n\cup\psi^m(X_n))=F\dot{\cup}K$ where $F$ is (relatively) closed in $X$ and $K\subseteq X$ is compact, then $K=\emptyset$. 
\end{proposition}

\begin{proof}
	We first assume that $X=\R^d$. From the hypothesis it follows that for each $n\in\N$ there is $m\in\N$ such that $B(0,n)$ and $\psi^m(B(0,n))$ are disjoint. Applying Proposition \ref{obvious topological facts} to $\psi^m$ we obtain that $\R^d\backslash\big(B(0,n)\cup\psi^m(B(0,n))\big)$ is connected. In particular, for every closed set $F\subseteq\R^d$ and each compact subset $K$ of $\R^d\backslash F$ we have
	$$\R^d\backslash\big(B(0,n)\cup\psi^m(B(0,n))\big)=F\dot{\cup}K\Rightarrow K=\emptyset.$$
	Thus, for $X=\R^d$ we can choose $X_n:=B(0,n), n\in\N$.
	
	Now let $X\subseteq\R^d$ be an arbitrary open subset homeomorphic to $\R^d$ via $\Phi:X\rightarrow\R^d$. Then $\Phi\circ\psi\circ\Phi^{-1}$ is a continuous, injective mapping on $\R^d$ with the run-away property. For $n\in\N$ let $X_n:=\Phi^{-1}(B(0,n))$ so that $(X_n)_{n\in\N}$ is an open, relatively compact exhaustion of $X$. Let $n,m\in\N$ be such that $X_n$ and $\psi^m(X_n)$ are disjoint and let $F\subseteq X$ be relatively closed and $K\subseteq X$ be compact such that
	$$F\dot{\cup}K=X\backslash\big(X_n\cup\psi^m(X_n)\big)=\Phi^{-1}\Big(\R^d\backslash\big(B(0,n)\cup(\Phi\circ\psi\circ\Phi)^m(B(0,n))\big)\Big).$$
	Since $\Phi(F)$ is a closed subset of $\R^d$ and $\Phi(K)$ is compact it follows together with
	$$\Phi(F)\dot{\cup}\Phi(K)=\R^d\backslash\big(B(0,n)\cup(\Phi\circ\psi\circ\Phi)^m(B(0,n))\big)$$
	and the case of $X=\R^d$ applied to $\Phi\circ\psi\circ\Phi^{-1}$ that $\Phi(K)=\emptyset$ hence $K=\emptyset$.
\end{proof}

\textit{Proof of Theorem \ref{dynamics on elliptic kernels}.} Since $P$ is elliptic the sheaf $C_P^\infty$ satisfies $(\mathscr{F}1)-(\mathscr{F}4)$ by Proposition \ref{f3 and f4 for kernel}. We can therefore invoke Theorem \ref{almost characterization of transitivity} in order to prove part a). If i) holds, i.e.\ if $C_{w,\psi}$ is weakly mixing, $C_{w,\psi}$ has obviously dense range and the rest of ii) follows from Theorem \ref{almost characterization of transitivity}. If ii) holds, it follows from Proposition \ref{dense range} that iii) is true. 

Next, if iii) holds, condition a) from part i) of Theorem \ref{almost characterization of transitivity} is fulfilled. Let $(X_n)_{n\in\N}$ be the open, relatively compact exhaustion of $X$ from Proposition \ref{preparation Lax-Malgrange}. By the injectivity of $\psi$ it follows from Brouwer's Invariance of Domain Theorem that b1) of part i) from Theorem \ref{almost characterization of transitivity} is satisfied, while b2) is satisfied since $\psi$ is run-away. Fix $n\in\N$ and let $m\in\N$ be such that $X_n$ and $\psi^m(X_n)$ are disjoint. It follows from Proposition \ref{preparation Lax-Malgrange} that it is not possible to decompose $X\backslash\Big(X_n\cup\psi^m(X_n)\Big)$ into a relatively closed subset of $X$ and a non-empty compact subset of $K$ which are disjoint. Since $P$ is elliptic it follows from the Lax-Malgrange Theorem (see e.g.\ \cite[Theorem 4.4.5 combined with the remark preceding Corollary 4.4.4 resp.\ with Theorem 8.6.1]{Hoermander} or \cite[Theorem 3.10.7]{Narasimhan}) that
$$\{u_{|X_n\cup\psi^m(X_n)};\,u\in C_P^\infty(X)\}$$
is dense in $C_P^\infty(X_n\cup\psi^m(X_n))$, i.e.\ that $r_X^{X_n\cup\psi^m(X_n)}$ has dense range. Thus, conditions a) and b) from part i) of Theorem \ref{almost characterization of transitivity} are satisfied so that by this theorem $C_{w,\psi}$ is weakly mixing. Thus i)-iii) are equivalent. If additionally $|w(x)|\leq 1$ and $C_{w,\psi}$ is hypercyclic it follows from Theorem \ref{almost characterization of transitivity} that iii) holds. Since trivially i) implies iv), a) is proved.

The proof of part b) is mutatis mutandis a repetition of the above arguments with the reference to Theorem \ref{almost characterization of transitivity} replaced by a reference to Theorem \ref{almost characterization of mixing}.\hfill$\square$\\  

In the remainder of this section we are going to characterize the dynamics for weighted composition operators on eigenspaces of the Cauchy-Riemann operator and Laplace operator respectively, i.e.\ on $C_P^\infty$ for the polynomial in $d=2$ variables $P(\xi)=\frac{1}{2}(\xi_1+i\xi_2)-\lambda$, resp.\ in $d$ variables $P(\xi)=\sum_{j=1}^d \xi_j^2-\lambda$ where in both cases $\lambda\in\C$ is arbitrary. We begin our considerations for these special operators by determining explicitly the combinations of symbols and weights which yield well-defined weighted composition operators on $C_P^\infty(X)$. 

\begin{proposition}\label{well-definedness on eigenspaces of laplace and cauchy-riemann}
	\hspace{2em}
	\begin{itemize}
		\item[a)] Let $\lambda\in\C$ and let $d=2$ as well as $P(\xi)=\frac{1}{2}(\xi_1+i\xi_2)-\lambda$. For $X\subseteq\R^2=\C$ open, $w:X\rightarrow\C$ and $\psi:X\rightarrow X$ smooth functions the following are equivalent.
		\begin{itemize}
			\item[i)] $C_{w,\psi}$ is well-defined on $C_P^\infty(X)$.
			\item[ii)] $w\partial_{\bar{z}}\psi=0$ and $P(\partial)w=-\lambda w\,\partial_{\bar{z}}\bar{\psi}$.
		\end{itemize}
	Moreover, if $C_{w,\psi}$ is well-defined on $C_P^\infty(X)$ it follows that for every $Y\subseteq X$ open and $f\in C^\infty(Y)$ we have
	$$P(\partial)\big(C_{w,\psi}(f)\big)=\partial_{\bar{z}}\bar{\psi}\, C_{w,\psi}\big(P(\partial)f\big).$$
		\item[b)] Let $\lambda\in\C$ and let $d\geq 2$ as well as $P(\xi)=\sum_{j=1}^d \xi_j^2-\lambda$. For $X\subseteq\R^d$ open, $w:X\rightarrow\C$ and $\psi:X\rightarrow X$ smooth functions the following are equivalent.
		\begin{itemize}
			\item[i)] $C_{w,\psi}$ is well-defined on $C_P^\infty(X)$.
			\item[ii)] For every $1\leq j\neq k\leq d$ it holds $w|\nabla \psi_j|^2=w|\nabla\psi_k|^2$ as well as $w\langle\nabla\psi_j,\nabla\psi_k\rangle=0$, $w\Delta\psi_j+2\langle\nabla w,\nabla\psi_j\rangle =0$, and $P(\partial)w=-\lambda w|\nabla\psi_1|^2$.
		\end{itemize}
		Moreover, if $C_{w,\psi}$ is well-defined on $C_P^\infty(X)$ it follows that for every $Y\subseteq X$ open and $f\in C^\infty(Y)$ we have
		$$P(\partial)\big(C_{w,\psi}(f)\big)=|\nabla\psi_1|^2\, C_{w,\psi}\big(P(\partial)f\big).$$
	\end{itemize}	
\end{proposition}

\begin{proof}
	We use the notation $f(\psi)$ instead of $f\circ\psi$ in order to slightly simplify notation. In order to prove a) it is straightforward to verify that for every $f\in C^\infty(X)$
	\begin{equation}\label{equation cr0)}
		P(\partial)\big(w\cdot f(\psi)\big)=(P(\partial)w)f(\psi)+w(\partial_{z}f)(\psi)\partial_{\bar{z}}\psi+w(\partial_{\bar{z}}f)(\psi)\partial_{\bar{z}}\bar{\psi}.
	\end{equation}
	Now assume that i) in part a) holds. Inserting $f=e_\zeta$ with $\zeta=(2\lambda,0)\in\C^2$ into equation (\ref{equation cr0)}) we obtain from $e_\zeta\in C_P^\infty(X)$ (recall that $e_\zeta(x)=\exp(\sum_{j=1}^2\zeta_j x_j)$) that
	$$0=\big(P(\partial)w+\lambda w(\partial_{\bar{z}}\psi+\partial_{\bar{z}}\bar{\psi})\big)e_\zeta$$
	so that
	\begin{equation}\label{equation cr a)}
		0=P(\partial)w+\lambda w(\partial_{\bar{z}}\psi+\partial_{\bar{z}}\bar{\psi}).
	\end{equation}
	Likewise, we derive from equation (\ref{equation cr0)}) by inserting $f=e_\eta$ with $\eta=(0,-2i\lambda)\in\C^2$ that
	\begin{equation}\label{equation cr b)}
		0=P(\partial)w-\lambda w(\partial_{\bar{z}}\psi-\partial_{\bar{z}}\bar{\psi}).
	\end{equation}
	Substracting equation (\ref{equation cr b)}) from equation (\ref{equation cr a)}) yields
	\begin{equation}\label{equation cr c)}
		0=\lambda w\partial_{\bar{z}}\psi
	\end{equation}
	while adding equations (\ref{equation cr a)}) and (\ref{equation cr b)}) gives
	\begin{equation}\label{equation cr d)}
		0=P(\partial)w+\lambda w\partial_{\bar{z}}\bar{\psi}.
	\end{equation}
	In case of $\lambda=0$ we evaluate equation (\ref{equation cr0)}) for $f(x)=x_1+ix_2$ which gives $w\partial_{\bar{z}}\psi=0$. In case of $\lambda\neq 0$ we have $w\partial_{\bar{z}}\psi=0$, too, by equation  (\ref{equation cr c)}) showing one half of ii). Additionally, evaluating equation (\ref{equation cr0)}) for arbitrary $f\in C_P^\infty(X)$ gives
	\begin{align*}
		\forall\,f\in C_P^\infty(X):\,0&=(P(\partial)w)f(\psi)+w(\partial_z f)(\psi)\partial_{\bar{z}}\psi+w(\partial_{\bar{z}}f)(\psi)\partial_{\bar{z}}\bar{\psi}\\&=(P(\partial)w)f(\psi)+\lambda w f(\psi) \partial_{\bar{z}}\bar{\psi}\\
		&=\big(\partial_{\bar{z}}w-(1-\partial_{\bar{z}}\bar{\psi})\lambda w\big)f(\psi).
	\end{align*}
	Inserting $e_\zeta$ with $\zeta$ as above into this equation gives $\partial_{\bar{z}}w-(1-\partial_{\bar{z}}\bar{\psi})\lambda w=0$ which proves that a) i) implies a) ii).
	
	On the other hand, if a) ii) is satisfied, it follows from equation (\ref{equation cr0)}) that for every $f\in C_P^\infty (X)$ we have
	$$P(\partial)\big(C_{w,\psi}(f)\big)=\partial_{\bar{z}}\bar{\psi}\, C_{w,\psi}\big(P(\partial)f\big)$$
	which proves a) i).
	
	To finish the proof of a), let $Y\subseteq X$ be open and assume that $C_{w,\psi}$ is well-defined on $C_P^\infty(X)$. Using a) ii) it is straightforward to derive - compare equation (\ref{equation cr0)})
	$$\forall\,f\in C^\infty(Y):P(\partial)\big(C_{w,\psi}(f)\big)=\partial_{\bar{z}}\bar{\psi}\, C_{w,\psi}\big(P(\partial)f\big).$$
	
	In order to prove b), we first notice that for $f\in C^\infty(X)$
	\begin{align}\label{equation 0}
		P(\partial)(w\cdot f(\psi))&=P(\partial)w\cdot f(\psi)+\sum_{l=1}^d\big(2\langle\nabla w,\nabla\psi_l\rangle+w\Delta\psi_l\big)\partial_l f(\psi)\nonumber\\
		&+w\big(\sum_{l=1}^d\sum_{m=1}^d(\partial_l\partial_m f)(\psi)\langle\nabla\psi_l,\nabla\psi_m\rangle\big).
	\end{align}
	We first show that b) i) implies b) ii). Inserting $f=e_{\zeta_j}, 1\leq j\leq d,$ with $\zeta_j=\sqrt{\lambda}(\delta_{j,l})_{1\leq l\leq d}$ for any root $\sqrt{\lambda}$ of $\lambda$ it follows from $e_{\zeta_j}\in C_P^\infty(X)$ that
	$$0=\Big(P(\partial)w+\sqrt{\lambda}\big(2\langle\nabla w,\nabla\psi_j\rangle+w\Delta\psi_j\big)+\lambda w|\nabla\psi_j|^2\Big)e_{\zeta_j}$$
	so that for every $1\leq j\leq d$
	\begin{align}\label{equation a)}
		0&=P(\partial)w+\sqrt{\lambda}\big(2\langle\nabla w,\nabla\psi_j\rangle+w\Delta\psi_j\big)+\lambda w|\nabla\psi_j|^2\nonumber\\
		&=\big(\Delta w-(1-|\nabla\psi_j|^2)\lambda w\big)+\sqrt{\lambda}\big(2\langle\nabla w,\nabla\psi_j\rangle+w\Delta\psi_j\big).
	\end{align}
	Analogously, inserting $f=e_{\tilde{\zeta}_j}, 1\leq j\leq d,$ with $\tilde{\zeta}_j=-\sqrt{\lambda}(\delta_{j,k})_{1\leq k\leq d}$ into equation (\ref{equation 0}) yields for $1\leq j\leq d$
	\begin{equation}\label{equation b)}
		0=\big(\Delta w-(1-|\nabla\psi_j|^2)\lambda w\big)-\sqrt{\lambda}\big(2\langle\nabla w,\nabla\psi_j\rangle+w\Delta\psi_j\big).
	\end{equation}
	Adding equations (\ref{equation a)}) and (\ref{equation b)}) gives
	\begin{equation}\label{equation f)}
		0=\Delta w-(1-|\nabla\psi_j|^2)\lambda w=P(\partial)w+\lambda|\nabla\psi_j|^2w,
	\end{equation}
	for every $1\leq j\leq d$ while substracting equation (\ref{equation b)}) from equation (\ref{equation a)}) gives
	\begin{equation}\label{equation c)}
		\forall\,1\leq j\leq d: 0=2\langle\nabla w,\nabla\psi_j\rangle+w\Delta\psi_j
	\end{equation}
	for $\lambda\neq 0$. In case of $\lambda=0$ we have that $f(x)=x_j\in C_P^\infty(X)$ and plugging this $f$ into (\ref{equation 0}) shows that (\ref{equation c)}) is also valid for $\lambda=0$.
	
	Next we insert $f=e_{\eta^\pm_{j,k}}, 1\leq j\neq k\leq d,$ into equation (\ref{equation 0}), where for $\alpha\in\C\backslash\{0,-\lambda\}$ $\eta_{j,k}^\pm=\sqrt{\lambda+\alpha}(\delta_{j,l})_{1\leq l\leq d}\pm i\sqrt{\alpha}(\delta_{k,l})_{1\leq l\leq d}$ resulting in
	\begin{align}\label{equation g)}
		0&=P(\partial)w+\sqrt{\lambda+\alpha}\big(2\langle\nabla w,\nabla \psi_j\rangle+w\Delta\psi_j\big)\nonumber\\
		& \pm i\sqrt{\alpha}\big(2\langle\nabla w,\nabla \psi_k\rangle+w\Delta\psi_k\big)\\
		& +w\big((\lambda+\alpha)|\nabla\psi_j|^2-\alpha|\psi_k|^2\pm 2i\sqrt{\alpha}\sqrt{\lambda+\alpha}\langle\nabla\psi_j,\nabla\psi_k\rangle\big).\nonumber
	\end{align}
	Substracting from the version with "+" of the above equation the version with "-" yields
	$$0=2i\sqrt{\alpha}\big(2\langle\nabla w,\nabla\psi_k\rangle+w\Delta\psi_k\big)+w4i\sqrt{\alpha}\sqrt{\lambda+\alpha}\langle\nabla\psi_j,\nabla\psi_k\rangle$$
	so taking into account equation (\ref{equation c)}) we derive
	\begin{equation}\label{equation d)}
		\forall\,1\leq j\neq k\leq d:\,0=w\langle\nabla\psi_j,\nabla\psi_k\rangle.
	\end{equation}
	Taking into account (\ref{equation c)}) and (\ref{equation d)}), equation (\ref{equation g)}) combined with (\ref{equation f)}) gives
	\begin{align}\label{equation e)}
		0&=\Delta w-(1-|\nabla\psi_j|^2)\lambda w+\alpha w\big(|\nabla\psi_j|^2-|\nabla\psi_k|^2\big)\\
		&=\alpha w\big(|\nabla\psi_j|^2-|\nabla\psi_k|^2\big).\nonumber
	\end{align}
	Since $\alpha\neq 0$, equation (\ref{equation e)}) together with equations (\ref{equation c)}) and (\ref{equation d)}) now give ii) in b).
	
	On the other hand, if ii) of b) is satisfied, equation (\ref{equation 0}) simplifies to
	\begin{align*}
		\forall\,f\in C^\infty(X):\,P(\partial)(w\cdot f(\psi))&=-\lambda w|\nabla\psi_1|^2\cdot f(\psi)+w|\nabla\psi_1|^2\big(\Delta f\big)(\psi)\\
		&=|\nabla\psi_1|^2w\big(P(\partial)f\big)(\psi),
	\end{align*}
	in particular $C_{w,\psi}$ is well-defined on $C_P^\infty(X)$ proving b) i).
	
	To finish the proof of b), let $Y\subseteq X$ be open and assume that $C_{w,\psi}$ is well-defined on $C_P^\infty(X)$. Using b) ii) it is straightforward to derive - compare equation (\ref{equation 0})
	$$\forall\,f\in C^\infty(Y): P(\partial)\big(C_{w,\psi}(f)\big)=|\nabla\psi_1|^2\,C_{w,\psi}\big(P(\partial)f\big).$$
\end{proof}

\begin{corollary}\label{dynamics on eigenspaces of the laplace operator}
	Let $d\geq 2$, $\lambda\in\C$ and let $P(\xi)=\sum_{j=1}^d\xi_j^2-\lambda$. Moreover, let $X\subseteq\R^d$ be homeomorphic to $\R^d$ and assume that the smooth mappings $w:X\rightarrow\C$ and $\psi:X\rightarrow X$ are such that $C_{w,\psi}$ is well-defined on $C_P^\infty(X)$.
	\begin{itemize}
		\item[a)]  For $C_{w,\psi}$ the following are equivalent.
		\begin{itemize}
			\item[i)] $C_{w,\psi}$ is weakly mixing on $C_P^\infty(X)$.
			\item[ii)] $w$ has no zeros, $\psi$ is injective as well as run-away and satisfies $\mbox{det}J\psi(x)\neq 0$ for each $x\in X$.
		\end{itemize}
		If additionally $|w(x)|\leq 1$ for all $x\in X$ the above are equivalent to
		\begin{itemize}
			\item[iii)] $C_{w,\psi}$ is hypercyclic on $C_P^\infty(X)$.
		\end{itemize}
		\item[b)]  For $C_{w,\psi}$ the following are equivalent.
		\begin{itemize}
			\item[i)] $C_{w,\psi}$ is mixing on $C_P^\infty(X)$.
			\item[ii)] $w$ has no zeros, $\psi$ is injective as well as strong run-away and $\mbox{det}J\psi(x)\neq 0$ for each $x\in X$.
		\end{itemize}
	\end{itemize}
\end{corollary}

\begin{proof}
	Because $C_{w,\psi}$ is well-defined on $C_P^\infty(X)$ it follows from Proposition \ref{well-definedness on eigenspaces of laplace and cauchy-riemann} b) that for all $Y\subseteq X$ open we have
	\begin{equation}\label{commuting}
		\forall\,f\in C^\infty(Y): P(\partial)\big(C_{w,\psi}(f)\big)=|\nabla\psi_1|^2\,C_{w,\psi}\big(P(\partial)f\big)
	\end{equation}
	Clearly, (\ref{commuting}) implies that $C_{w,\psi}$ acts locally on $C_{w,\psi}(X)$ so that by Theorem \ref{dynamics on elliptic kernels} we only have to show that a) ii) implies a) i) and that b) ii) implies b) i), respectively. This will be done once we have shown that under a) ii), respectively b) ii), it holds that for each $m\in\N_0$ and every open, relatively compact $Y\subseteq\psi^m(X)$ we have
	$$r_X^{(\psi^m)^{-1}(Y)}(C_P^\infty(X))\subseteq\overline{\big(C_{w,\psi,(\psi^{m-1})^{-1}(Y)}\circ\ldots\circ C_{w,\psi,Y}\big)(C_P^\infty(Y))}.$$
	If a) ii), respectively b) ii), holds we have $\mbox{det}J\psi(x)\neq 0$ for each $x\in X$ and we conclude that $\psi^m(X)$ is open and $(\psi^{m})^{-1}:\psi^m(X)\rightarrow X$ is smooth for every $m\in\N$ as well as $|\nabla\psi_1(x)|^2\neq 0$ for all $x\in X$. Moreover, (\ref{commuting}) implies for every open set $Y\subseteq \psi^m(X)$ and every $f\in C^\infty((\psi^{m})^{-1}(Y))$ that
	$$P(\partial)f=|\nabla\psi_1|^{2m}\prod_{j=0}^{m-1}w(\psi^j(\cdot))\Big(\big(P(\partial)[\big(\frac{f}{\prod_{j=0}^{m-1}w(\psi^j(\cdot))}\big)\circ(\psi^m)^{-1}]\big)\circ\psi^m\Big)$$
	which in turn yields
	$$P(\partial)\Big(\big(\frac{f}{\prod_{j=0}^{m-1}w(\psi^j(\cdot))}\big)\circ(\psi^m)^{-1}\Big)=\Big(\frac{P(\partial)f}{|\nabla\psi_1|^{2m}\prod_{j=0}^{m-1}w(\psi^j(\cdot))}\Big)\circ(\psi^m)^{-1}.$$
	Therefore, for each $m\in\N$, for every open $Y\subseteq \psi^m(X)$, and every $f\in C_P^\infty((\psi^{m})^{-1}(Y))$ it follows
	$$\big(\frac{f}{\prod_{j=0}^{m-1}w(\psi^j(\cdot))}\big)\circ(\psi^m)^{-1}\in C_P^\infty(Y)$$
	so that by Remark \ref{considering the almost characterization} iii) we have
	$$C_P^\infty((\psi^{m})^{-1}(Y))\subseteq\big(C_{w,\psi,(\psi^{m-1})^{-1}(Y)}\circ\ldots\circ C_{w,\psi,Y}\big)(C_P^\infty(Y)).$$
	In particular,
	$$r_X^{(\psi^m)^{-1}(Y)}(C_P^\infty(X))\subseteq\overline{\big(C_{w,\psi,(\psi^{m-1})^{-1}(Y)}\circ\ldots\circ C_{w,\psi,Y}\big)(C_P^\infty(Y))}$$
	which is all that had to be shown.
\end{proof}

We close this section by applying Theorem \ref{dynamics on elliptic kernels} to characterize dynamics of weighted composition operators on eigenspaces of the Cauchy-Riemann operator. For the case $\lambda=0$ the equivalence of iii) and iv) is established in \cite{YousefiRezaei} while the equivalence of iii) and i) for $\lambda=0$ was also considered in \cite{Bes}, both without any restriction on the range of the weight $w$.

\begin{corollary}\label{dynamics on eigenspaces of the cauchy-riemann operator}
	Let $d=2$, $\lambda\in\C$ and let $P(\xi)=\frac{1}{2}\big(\xi_1+i\xi_2\big)-\lambda$. Moreover, let $X\subseteq\R^2$ be homeomorphic to $\R^2$ and assume that the smooth mappings $w:X\rightarrow\C$ and $\psi:X\rightarrow X$ are such that $C_{w,\psi}$ is well-defined on $C_P^\infty(X)$. Then, the following are equivalent.
	\begin{itemize}
		\item[i)] $C_{w,\psi}$ is mixing on $C_P^\infty(X)$ 
		\item[ii)] $C_{w,\psi}$ is weakly mixing on $C_P^\infty(X)$.
		\item[iii)] $w$ has no zeros, $\psi$ is injective, holomorphic, and has no fixed point.
	\end{itemize}
	If additionally $|w(x)|\leq 1$ for all $x\in X$ the above are equivalent to
	\begin{itemize}
		\item[iv)] $C_{w,\psi}$ is hypercyclic on $C_P^\infty(X)$.
	\end{itemize}
\end{corollary}

\begin{proof}
	We identify $\R^2$ with $\C$. Because $C_{w,\psi}$ is well-defined on $C_P^\infty(X)$ it follows from Proposition \ref{well-definedness on eigenspaces of laplace and cauchy-riemann} b) that for all $Y\subseteq X$ open we have
	\begin{equation}\label{commuting2}
	\forall\,f\in C^\infty(Y):\,P(\partial)\big(C_{w,\psi}(f)\big)=\partial_{\bar{z}}\bar{\psi}\,C_{w,\psi}\big(P(\partial) f\big)
	\end{equation}
	Equation (\ref{commuting2}) implies that $C_{w,\psi}$ acts locally on $C_{w,\psi}(X)$.
	
	Clearly, i) implies ii) and by Theorem \ref{dynamics on elliptic kernels}, if ii) holds, in particular, $w$ has no zeros, $\psi$ is injective and run-away. Especially, $\psi$ has no fixed point. Since $w$ has no zeros it follows from Proposition \ref{well-definedness on eigenspaces of laplace and cauchy-riemann} a) ii) that $\psi$ is holomorphic so that ii) implies iii).
	
	Next, if iii) holds, the composition operator $C_\psi$ is in particular a well-defined continuous linear operator on $\mathscr{H}(X)$, the holomorphic functions on $X$. Since $X$ is homeomorphic to $\C$, $X$ is a simply connected domain in $\C$ and because $\psi$ is injective and has no fixed point, it follows from \cite[Proof of Theorem 3.1]{Bes} that $\psi$ is strong run-away. From the injectivity of the holomorphic mapping $\psi$ we conclude that $0\neq |\psi'(z)|=\mbox{det}J\psi(z)$ for all $z\in X$ (see e.g\ \cite[Theorem 10.33]{Rudin}). In view of Theorem \ref{dynamics on elliptic kernels} we only have to show that for each $m\in\N_0$ and every open, relatively compact $Y\subseteq\psi^m(X)$ we have
	$$r_X^{(\psi^m)^{-1}(Y)}(C_P^\infty(X))\subseteq\overline{\big(C_{w,\psi,(\psi^{m-1})^{-1}(Y)}\circ\ldots\circ C_{w,\psi,Y}\big)(C_P^\infty(Y))}.$$
	From $0\neq \mbox{det}J\psi(x)=|\partial_z\psi(x)|^2$ for all $x\in X$ together with $\partial_{\bar{z}}\bar{\psi}=\overline{\partial_z\psi}$ it follows that $\partial_{\bar{z}}\bar{\psi}$ has no zeros in $X$. Thus, equation (\ref{commuting2}) can be used as equation (\ref{commuting}) in the proof of Corollary \ref{dynamics on eigenspaces of the laplace operator} to prove that i) holds.
	
	Finally, if $|w(x)|\leq 1$ for all $x\in X$ it follows from Theorem \ref{dynamics on elliptic kernels} that ii) and iv) are equivalent which completes the proof. 
\end{proof}

Motivated by the previous result we close this paper with two open problems. While the first one is concerned with the general abstract setting the second one aims at a more manageable characterization of hypercyclicity/mixing for weighted composition operators on eigenspaces of the Laplace operator.

\begin{problem}
1. Are the additional assumptions in Proposition \ref{necessary transitive} v) superfluous to prove that $\psi$ is run-away whenever $C_{w,\psi}$ is transitive? If this is the case, the additional assumption on $w$ (or the sheaf $\mathscr{F}$) in Theorem \ref{almost characterization of transitivity} can be removed so that i)-iv) in Theorem \ref{dynamics on elliptic kernels} are equivalent as well as i)-iii) in Corollary \ref{dynamics on eigenspaces of the laplace operator} a) and i)-iv) in Corollary \ref{dynamics on eigenspaces of the cauchy-riemann operator}, without the additional assumption on the range of the weight $w$.

2. Let $X\subseteq\R^d$ be homeomorphic to $\R^d$, $\lambda\in\C$, and $P(\xi)=\sum_{j=1}^d\xi_j^2-\lambda$. Characterize those $w\in C^\infty(X), |w|\leq 1$ and smooth $\psi:X\rightarrow X$ such that
\begin{itemize}
	\item[i)] $\forall\,1\leq j\neq k\leq d: |\nabla \psi_j|^2=|\nabla\psi_k|^2\mbox{ and }\langle\nabla\psi_j,\nabla\psi_k\rangle=0$,
	\item[ii)] $\forall\,1\leq j\leq d:\,w\Delta\psi_j+2\langle\nabla w,\nabla\psi_j\rangle =0\mbox{ and }\Delta w-\lambda w=-\lambda|\nabla\psi_1|^2$,
	\item[iii)] $\psi$ is run-away.
\end{itemize}
Are there hypercyclic weighted composition operators on $C_P^\infty(X)$ which are not mixing?
\end{problem}

\subsection*{Acknowledgements}
The author would like to thank J.\ Wengenroth for pointing out Alexander Duality, \cite[Theorem 3.44]{Hatcher}, in connection with the version of the Jordan-Brouwer Separation Theorem which is used in this article. Moreover, the author is indebted to one of the anonymous referees for pointing out a gap in the proof of the previous version of Proposition \ref{necessary transitive} v).

\end{document}